\documentclass[oneside]{amsart}
\usepackage{amssymb}
\usepackage{amsmath}
\usepackage{amsfonts}
\usepackage{amsthm}
\usepackage{hyperref}
\usepackage{amscd}
\usepackage{color}
\usepackage{enumerate}
\usepackage{xypic}
\usepackage{tikz-cd}

\newcommand{\Aff}{\mathbb{A}}
\newcommand{\ZZ}{\mathbb{Z}}
\newcommand{\KK}{\mathbb{K}}
\newcommand{\RR}{\mathbb{R}}
\newcommand{\PP}{\mathbb{P}}

\newcommand{\QQ}{\mathbb{Q}}

\newcommand{\D}{\mathcal{D}}

\newcommand{\CO}{\mathcal{O}}
\newcommand{\bx}{{\bf x}}

\newcommand{\fv}{\mathfrak{v}}

\renewcommand{\H}{\mathcal{H}}
\newcommand{\G}{\mathcal{G}}
\newcommand{\mcL}{\mathcal{L}}
\newcommand{\mcP}{\mathcal{P}}
\newcommand{\F}{\mathcal{F}}

\DeclareMathOperator{\trop}{Trop}
\DeclareMathOperator{\Pic}{Pic}
\DeclareMathOperator{\im}{Im}
\DeclareMathOperator{\In}{In}

\DeclareMathOperator{\proj}{Proj}
\DeclareMathOperator{\spec}{Spec}
\DeclareMathOperator{\Div}{Div}
\DeclareMathOperator{\Cl}{Cl}
\DeclareMathOperator{\cox}{Cox}
\DeclareMathOperator{\rank}{rank}
\DeclareMathOperator{\ord}{ord}
\DeclareMathOperator{\id}{id}
\DeclareMathOperator{\Gr}{Gr}

\newtheorem{thm}{Theorem}[section]
\newtheorem{prop}[thm]{Proposition}
\newtheorem{lemma}[thm]{Lemma}
\newtheorem{cor}[thm]{Corollary}

\theoremstyle{definition}

\newtheorem{defn}[thm]{Definition}
\newtheorem{rem}[thm]{Remark}

\newtheorem{ex}[thm]{Example}

\title{Rational Complexity-One $T$-Varieties are Well-Poised}

\author{Nathan Ilten}
\address{Department of Mathematics, Simon Fraser University,
8888 University Drive, Burnaby BC V5A1S6, Canada}
\email{nilten@sfu.ca}

\author{Christopher Manon}
\address{Department of Mathematical Sciences, George Mason University, 4400 University Drive, MS:  3F2
Exploratory Hall, room 4400 Fairfax, Virginia  22030}
\email{cmanon@gmu.edu}
\begin{document}

\begin{abstract}
	Given an affine rational complexity-one $T$-variety $X$, we construct an explicit embedding of $X$ in affine space $\Aff^n$. We show that this embedding is well-poised, that is, every initial ideal of $I_X$ is a prime ideal, and we determine the tropicalization $\trop(X^\circ)$. We then study valuations of the coordinate ring $R_X$ of $X$ which respect the torus action, showing that for full rank valuations, the natural generators of $R_X$ form a Khovanskii basis. This allows us to determine Newton-Okounkov bodies of rational projective complexity-one $T$-varieties, partially recovering (and generalizing) results of Petersen. We apply our results to describe all integral special fibers of $\KK^*\times T$-equivariant degenerations of rational projective complexity-one $T$-varieties, generalizing a result of S\"u\ss{} and the first author.
\end{abstract}
\maketitle
\section{Introduction}

Let $X \subset \Aff^n$ be an irreducible affine variety  over a trivially valued, algebraically closed field  $\KK$.   There is a corresponding presentation of the coordinate ring $R_X$ of $X$ by the polynomial algebra $\KK[\bx]$, where $\bx = \{x_1, \ldots, x_n\}$:

$$
\begin{CD}
0 @>>> I_X @>>> \KK[\bx] @>\pi>> R_X @>>> 0.\\
\end{CD}
$$

\bigskip
\noindent
Here $I_X$ is the prime ideal of polynomials in $\KK[\bx]$ which vanish on $X$.    We are concerned with the \emph{tropical geometry} of the variety $X$. We consider the tropical variety $\trop(X^\circ) \subset \RR^n$ associated to the very affine variety $X^\circ=X\cap(\KK^*)^n$ (\cite[Definition 3.2.1]{tropical} and \S \ref{sec:tropical}).  We assume that $X^\circ \neq \emptyset$.  Recall that to each point $w \in \trop(X^\circ)$ there is an associated initial ideal $\In_{w}(I_X) \subset \KK[\bx]$.  The purpose of this paper is in part to introduce the following property for affine embeddings:  

\begin{defn}
The embedding of $X$ into $\Aff^n$ is said to be \emph{well-poised} if $\In_{w}(I_X)$ is a prime ideal for all $w \in \trop(X^\circ)$. 
\end{defn}

Well-poised embeddings can tell us much about the geometry of $X$. For example, when $I_X$ is a homogeneous ideal, all of the Gr\"obner degenerations associated to the tropical points $w \in \trop(I_X)$ of a well-poised embedding are reduced, irreducible varieties.  In particular, a generic tropical point defines a (possibly non-normal) toric degeneration of $X$.

Given the connection with toric geometry, it is not surprising that affine toric varieties themselves always have well-poised embeddings.   It can be arranged that the coordinate ring of an affine toric variety $X$ be presented by a prime binomial ideal $J$ \cite[Proposition 1.1.9]{CLS}. With respect to this embedding, $\trop(X^\circ)$ is a rational vector space of dimension $\dim X$, and every initial ideal is $J$ itself.   

A more exotic example of a well-poised embedding is provided by the Pl\"ucker embedding of the Grassmannian variety of $2$-planes:
\[\Gr(2,r) \subset \mathbb{P}^{\binom{r}{2}-1},\]
 or rather, the embedding of the affine cone of this variety in $\Aff^{\binom{r}{2}}$.   The tropical variety of this embedding is known as the tropical Grassmannian, and is described in the seminal paper \cite{Speyer-Sturmfels} of Speyer and Sturmfels.  The toric degenerations defined by the initial ideals of the tropical Grassmannian are also interesting objects. They have appeared in the study of integrable systems \cite{Nishinou-Nohara-Ueda} \cite{Howard-Manon-Millson}, and invariant theory \cite{Howard-Millson-Snowden-Vakil}.  The variety $\Gr(2,4)$ is an instructive special case.  This variety is cut out by a single equation $p_{12}p_{34} - p_{13}p_{24} + p_{14}p_{23}$ among the six Pl\"ucker variables $p_{ij}, 1 \leq i< j \leq 4$.   There are three ways to drop a monomial from this equation in order to make an irreducible binomial; these correspond to the three maximal cones of the tropical variety $\trop(\Gr(2,4)^\circ)$.

We can generalize the description of the case $\Gr(2,4)$ to a special type of hypersurface. Consider a polynomial ring $\KK[\bx]$ and any partition $\{S_1, \ldots, S_m\}$ of $[n]$.  This  defines an irreducible polynomial $f = \sum_{i=1}^m \prod_{j \in S_i}x_j$.  Any initial form extracted from this polynomial is likewise irreducible, so it follows that the embedded hypersurface cut out by $f$ is well-poised.  This last example is also interesting because when $m = 3$, the variety $V(f) \subset \Aff^n$ admits an effective action of an algebraic torus $T$ of dimension one less than $\dim V(f)$.  

 In general, a complexity-$k$ $T$-variety is a normal, irreducible variety $X$ equipped with an effective action by an algebraic torus $T$ such that $\dim T=\dim X-k$. By considering the action of a subtorus, any complexity-$k$ $T$-variety can be viewed as a complexity-$\ell$ $T$-variety for $\ell>k$.  In this sense, complexity-one $T$-varieties are a natural generalization of normal toric varieties, namely $T$-varieties of complexity-zero.   Our first theorem generalizes the observation that affine toric varieties possess well-poised embeddings to a larger class of $T$-varieties: 

\begin{thm}[Theorem \ref{thm:poised}]\label{main}
Every affine rational complexity-one $T$-variety has an equivariant embedding $X\hookrightarrow \Aff^n$ which is well-poised.
\end{thm}

To prove this theorem, for any affine rational complexity-one $T$-variety $X$, we construct an explicit family of embeddings, which we call \emph{semi-canonical embeddings}. Such an embedding is canonically determined by the geometry of $X$ up to equivariant isomorphism and the action of $(\KK^*)^n$.

The main idea behind the semi-canonical embedding is to embed $X$ in such a fashion such that the geometry of $X^\circ$ becomes as simple as possible. Under this embedding $X\subset \Aff^n$, the intersection of $X$ with $(\KK^*)^n$ is simply the product of the torus $T$ with $L^\circ=L\cap (\KK^*)^m$, where $L$ is a certain line embedded in projective space $\PP^m$ meeting the dense torus $(\KK^*)^m$.  Furthermore, the ideal $I_X$ describing a semi-canonical embedding $X\subset \Aff^n$ can be determined from the quasi-combinatorial data describing the geometry of $X$. We describe this embedding in detail in \S \ref{sec:embed} and \S \ref{sec:ideal}, and show in \S \ref{sec:trop} that every semi-canonical embedding is in fact well-poised.

Since the geometry of $X^\circ$ is especially simple, so is that of $\trop(X^\circ)$: modulo lineality space, $\trop(X^\circ)$ is just the tropical line $\trop(L^\circ)$. In particular, we are able to describe a tropical basis for $X^\circ$, see Corollary \ref{cor:tropbasis}.

When an embedding of $X$ is well-poised, work in \cite[Section 10]{Gubler-Rabinoff-Werner} implies the existence of a section to the tropicalization map from the Berkovich skeleton of $X$ to its tropicalization.  Our results here imply the existence of such a section over the tropicalization of any semi-canonical embedding.  The presence of a large torus action gives this result a similar flavor to work of Draisma and Postinghel in \cite{Draisma-Postinghel}. 

In the second half of the paper, we study the relationship between semi-canonical embeddings, higher rank valuations, and the theory of Newton-Okounkov bodies. For any full rank valuation $\fv$ of the coordinate ring $R_X$, its image is a semigroup $S(R_X,\fv)$. A set of generators of $R_X$ is a \emph{Khovanskii basis} if their images generate $S(R_X,\fv)$, see \S\ref{sec:val}. 
Our second main result shows the following:
\begin{thm}[Theorem \ref{thm:khovanskii}]
Let $X$ be an affine rational complexity-one $T$-variety, and $\fv$ a full rank valuation which is homogeneous with respect to the grading induced by the $T$-action. Then $X$ has a semi-canonical embedding for which the corresponding generators of $R_X$ form a Khovanskii basis.
  \end{thm}
\noindent Note that conversely, Theorem \ref{main} and \cite[\S 4]{Kaveh-Manon-NOK} imply that for any semi-canonical embedding of $X$, the corresponding generators of $R_X$ do form a Khovanskii basis for some full rank homogeneous valuation $\fv$.

Our second main result has a number of  consequences. We are able to give explicit descriptions of the value semigroups $S(R_X,\fv)$ for homogeneous valuations (Corollary \ref{cor:cone}). It turns out that all such valuations can be constructed as \emph{weight valuations} using the machinery of Kaveh and the second author, see \S\ref{sec:tropval}, Theorem \ref{thm:chris}, and \cite[\S 4]{Kaveh-Manon-NOK}.

Newton-Okounkov bodies of complexity $1$ $T$-varieties have been studied by Petersen in \cite{petersen}, where he gives a construction of Newton-Okounkov polytopes for any projective, complexity-one $T$-variety $Y$ with respect to valuations obtained from  $T$-invariant flags. As a corollary of our results mentioned above, we recover Petersen's description in the case when $Y$ is rational, and show that in fact, the Newton-Okounkov body with respect to \emph{any} homogeneous valuation is of the form described by Petersen, see \S\ref{sec:petersen}. Petersen also shows that the \emph{global} Newton-Okounkov bodies for complexity-one $T$-varieties are polyhedral; we easily recover this result again in the case of rational $T$-varieties, see \S \ref{sec:global}. 

As a further application of our second main theorem, we determine all integral special fibers of $\KK^*\times T$-equivariant degenerations of rational projective complexity-one $T$-varieties, see Theorem \ref{thm:degen}. In \cite{kstab}, S\"u\ss{} and the first author described all \emph{normal} irreducible special fibers of such degenerations. Our theorem generalizes this result to allow for non-normal special fibers.

There is a large amount of literature devoted to studying the \emph{algebra} of coordinate rings and ideals of toric varieties, for example, quadratic generation and Gr\"obner bases (e.g. \cite{quadratic}), Koszulness (e.g. \cite{koszul}), and free resolutions (e.g. \cite{syzygies}). We hope that the  introduction of semi-canonical presentations for the coordinate rings of affine rational complexity-one $T$-varieties will lead to a similar study of the algebra of such $T$-varieties. 

We now describe the structure of the remainer of this paper. In \S \ref{sec:embed}, we recall basics about $T$-varieties, construct our semi-canonical embeddings for affine rational complexity-one $T$-varieties, relate this to previous constructions, and discuss the projective case. We then precisely describe the ideal of a semi-canonical embedding in \S \ref{sec:ideal}. In \S\ref{sec:groebner} we recall basics of Gr\"obner theory and tropical geometry, and then proceed to analyze the initial ideals arising for semi-canonical embeddings, proving our Theorem \ref{main}. In \S\ref{sec:part2}, we move on to a discussion of valuations on coordinate rings of rational complexity-one $T$-varieties and their connections to semi-canonical embeddings. Finally, we apply our results in \S\ref{sec:test} to classify integral special fibers  of $\KK^*\times T$-equivariant degenerations of rational projective complexity-one $T$-varieties.

\section{Semi-Canonical Embeddings}\label{sec:embed}
\subsection{$T$-Varieties}\label{sec:tvar}
Let $T$ be an algebraic torus, with character lattice $M$ and co-character lattice $N$.
Recall that a \emph{$T$-variety} is a normal variety $X$ equipped with an effective action $T\times X\to X$. The \emph{complexity} of $X$ is $\dim X-\dim T$.  Such varieties may be described in terms of a `quotient' variety $Y$ of dimension equal to the complexity, equipped with some combinatorial data \cite{pdiv,dfan,tsurvey}. 

We briefly survey this correspondence for affine $T$-varieties. Let $Y$ be any normal variety, and fix a pointed, polyhedral cone $\sigma$ in $N_\QQ=N\otimes \QQ$. A \emph{polyhedral divisor} on $Y$ with tailcone $\sigma$ is a formal finite sum
\[
\D=\sum_{i=0}^m \Delta_i\cdot P_i
\]
where the $P_i$ are distinct prime divisors on $Y$, and the coefficients $\Delta_i$ are either polyhedra in $N_\QQ$ with tailcone $\sigma$, or the empty set.. Recall that the \emph{tailcone} of a polyhedron $\Delta\subset N_\QQ$ is the set of all $v\in N_\QQ$ such that $\Delta+v\subseteq \Delta$.

The polyhedral divisor $\D$ induces a piecewise linear convex map
\begin{align*}
\D:\sigma^\vee &\to \Div_{\QQ\cup\infty} Y\\
u&\mapsto \sum_i \Delta_i(u)\cdot P_i
\end{align*}
where 
\begin{align*}
\Delta_i(u):= \min_{v\in\Delta_i}\langle v,u\rangle
\end{align*}
and $\Delta_i(u)=\infty$ if $\Delta_i=\emptyset$.
 We use this to construct a $T$-scheme
\[
X(\D)=\spec \bigoplus_{u\in M\cap\sigma^\vee} H^0(Y,\CO(\D(u)))
=\spec \bigoplus_{u\in M\cap\sigma^\vee} H^0(Y,\CO(\lfloor \D(u)\rfloor)).
\]
Here an $\infty$-coefficient in a divisor means that we allow poles with arbitrary order along the corresponding prime divisor. 
If one imposes certain positivity conditions on $\D$, $X(\D)$ is actually a $T$-variety, and every affine $T$-variety can be constructed in this fashion \cite[Theorems 3.1 and 3.4]{pdiv}. Note that in \cite{pdiv}, these theorems are only stated when the characteristic of $\KK$ is zero. However, the proofs of these theorems should go through essentially verbatim in positive characteristic as well. We are only interested in the complexity-one case, which is dealt with explicitly in \cite{langlois}.

In this paper, we are primarily concerned with rational, complexity-one, affine $T$-varieties. In this situation, we can take the quotient $Y$ to be $\PP^1$, and the positivity condition on $\D$ is exactly that $\sum \Delta_i\subsetneq \sigma$ \cite[Example 2.12]{pdiv}.
In the following, we will show how to embed such a variety $X$ equivariantly in a particular toric variety. We will call this embedding a \emph{semi-canonical embedding} of $X$. The ambient toric variety is uniquely determined by $X$ up to equivariant isomorphism.

\subsection{Affine Embeddings}\label{sec:affine}
Suppose that we are given a rational, complexity-one affine $T$-variety $X$, which we have described as in \S\ref{sec:tvar} in terms of a polyhedral divisor \[\D=\sum_{i=0}^m \Delta_i P_i\]
 on $\PP^1$.
As before, the $\Delta_i$ have common pointed tailcone $\sigma$ (or equal $\emptyset$), and  $\sum \Delta_i\subsetneq \sigma$.
We will show how to embed $X=X(\D)$ into an affine toric variety.

 From the above data, we construct a pointed polyhedral cone in $N_\QQ\times\QQ^m$:
\[
C=\QQ_{\geq 0}\cdot \left(\{\Delta_i\times e_i\}_{i=0}^m\cup (\sigma\times 0)\right)
\]
where $e_1,\ldots,e_m$ is the standard basis of $\QQ^m$, and $e_0=-\sum_{i\neq 0} e_i$.
Note that $C$ is pointed, hence $C^\vee$ is full-dimensional. 

For any line $L\subset \PP^m=\proj \KK[y_0,\ldots,y_m]$ intersecting the torus $(\KK^*)^m$, let $Q_i$ be the restriction to $L$ of $V(y_i)$. 
We construct the following objects:
\begin{align*}
A(L)&=\bigoplus_{u\in\sigma^\vee\cap M} H^0\left(L,\CO\left(\sum_i \lfloor \Delta_i(u) \rfloor \cdot Q_i \right)\right)\cdot \chi^u\\
A(C)&=\KK[C^\vee\cap (M\times \ZZ^m)]=\bigoplus_{(u,v)\in C^{\vee}\cap M\times \ZZ^m} \KK\cdot \chi^{(u,v)}\\
X(L)&=\spec A(L)\qquad X(C)=\spec A(C).
\end{align*}
The functions $\chi^u$ are characters on the torus $T$, and keep track of the multidegree of homogeneous elements of $A(L)$; the functions $\chi^{(u,v)}$ are characters on the bigger torus $T\times (\KK^*)^m$ and keep track of the multidegree of homogeneous elements of $A(C)$.

\begin{rem}\label{rem:choice}
Our complexity-one $T$-variety $X(\D)$ is equivariantly isomorphic to $X(L)$ for an appropriate choice of $L$. Indeed, choose any linear embedding of $\PP^1$ into $\PP^m$ such that the divisor $V(y_i)$ on $\PP^m$ pulls back to the point $P_i\in \PP^1$. Taking $L_\D$ to be the image of $\PP^1$ under this embedding, the pullback to $\PP^1$ of $Q_i$ is exactly $P_i$. It follows by construction that $X(\D)$ is equivariantly isomorphic to $X(L_\D)$.

Concretely, such an embedding $\PP^1\to \PP^m$ can be constructed as follows. Representing each point $P_i\in\PP^1$ by a choice of homogeneous coordinates $(p_{i0}:p_{i1})$, we can embed $\PP^1$ into $\PP^m$ via 
\begin{align*}
\phi:\PP^1&\to \PP^m\\
(s_0:s_1)&\mapsto (s_0p_{00}-s_1p_{01}:s_0p_{10}-s_1p_{11}:\ldots:s_0p_{m0}-s_1p_{m1}).
\end{align*}
Note that different choices $(p_{i0}:p_{i1})$ representing the points $P_i$ correspond to acting on $L_\D$ by $(\KK^*)^m$.
\end{rem}

For any line $L\subset \PP^m$, we will show that $X(L)$ comes with a natural closed embedding in the toric variety $X(C)$. 
Let $I(L)$ be the ideal of $L^\circ=L\cap (\KK^*)^m$ in $(\KK^*)^m$. We define $\widetilde{I}(L)$ to be the ideal generated by 
\[
\{f\cdot \chi^{u}\ |\ f\in I(L)\ \textrm{and}\ f\cdot \chi^u\in A(C)\}_{u\in M}.
\] 

\begin{thm}\label{thm:embed}
There is a natural exact sequence
\[
\begin{tikzcd}
0\arrow{r} & \widetilde{I}(L)_u \arrow{r}& A(C)_u \arrow{r}& A(L)_u  \arrow{r} & 0
\end{tikzcd}
\]
for every $u\in M\cap \sigma^\vee$, and hence an embedding of $X(L)$ in $X(C)$ with ideal $\widetilde{I}(L)$.
\end{thm}
\begin{defn}
	Fix any rational affine complexity-one $T$-variety $X=X(\D)$. A \emph{semi-canonical embedding} of $X$ is an affine embedding induced by the embedding $X(L_\D)\subset X(C)$, composed with an equivariant affine embedding of $X(C)$, where $L_\D$ is any line in $\PP^m$ as in Remark \ref{rem:choice}.
\end{defn}
\begin{proof}[Proof of Theorem \ref{thm:embed}]
It is straightforward to check that 
\[
A(C)=\bigoplus_{u\in\sigma^\vee\cap M} H^0\left(\PP^m,\CO\left(\sum_i \lfloor \Delta_i(u) \rfloor \cdot V(y_i) \right)\right)\cdot \chi^u.
\]
We thus let the map $A(C)\to A(L)$ be the map induced by restriction of sections. For degree $u\in M$, the degree $u$ piece is just
\[
\begin{tikzcd}
 H^0\left(\PP^m,\CO\left(D \right)\right)
\arrow{r} &
 H^0\left(L,\CO_L\left(D \right)\right) 
\end{tikzcd}
\]
where
\[
 D=\sum_i \lfloor  \Delta_i(u) \rfloor \cdot V(y_i).
\]
This is clearly surjective.

It remains to check that the kernel of this map is the degree $u$ part of $\widetilde{I}(L)$. Note that $f\cdot \chi^{u}\in A(C)$ if and only if
\[ f\in H^0\left(\PP^m,\CO\left(D \right)\right).\]
Now, if we take such an $f$ that also vanishes on $L\cap (\KK^*)^m$, then clearly it restricts to zero on $L$, so $\widetilde{I}(L)$ is contained in the kernel. On the other hand, sections   
\[ f\in H^0\left(\PP^m,\CO\left(D \right)\right)\]
restricting to $0$ on $L$ are rational functions which vanish on $L$. Furthermore, since $D$ is supported outside the torus, they are regular on $(\KK^*)^m$. Hence, for such $f$, $f\cdot \chi^u \in \widetilde{I}(L)$ as desired.
\end{proof}

\begin{ex}[$D_6$ singularity]\label{ex:d6}
We consider $X$ to be a normal surface singularity of type $D_6$. This can be described by the polyhedral divisor 
\[\D=\Delta_0\cdot\{0\}+\Delta_1\cdot\{1\}+\Delta_2\cdot\{\infty\}\]
on $\PP^1$, where
\begin{align*}
\Delta_0=[3/2,\infty)\\
\Delta_1=[-1/2,\infty)\\
\Delta_2=[-1/2,\infty)
\end{align*}
The cone $C$ is generated by the columns of 
\[
\left(\begin{array}{c c c}
3 & -1 & -1\\
-2 & 2 & 0\\
-2 & 0 & 2
\end{array}\right)
\]
and the dual cone is generated by the columns of
\[
\left(\begin{array}{c c c}
2 & 2 & 2\\
1 & 1 & 2\\
1 & 2 & 1
\end{array}\right).
\]
A Hilbert basis for $C^\vee\cap(M\times\ZZ^2)$ is given by the columns of
\[
 \bgroup\begin{pmatrix}3&
      2&
      2&
      2\\
      2&
      1&
      2&
      1\\
      2&
      2&
      1&
      1\\
      \end{pmatrix}\egroup.\]
This means that the toric variety $X(C)$ is cut out by a single binomial: $x_1^2=x_2x_3x_4$.
This corresponds to the linear relation between the above Hilbert basis elements. 

Now, we can choose $L$ so that the ideal $I$ is generated by $f=\chi^{(1,0)}+\chi^{(0,1)}+1$.
To calculate $\widetilde{I}(L)$, we check: for which $(u,v_1,v_2)$ are the columns of 
\[
\left(\begin{array}{c c c}
u & u & u\\
v_1+1 & v_1 & v_1\\
v_2 &  v_2+1& v_2
\end{array}\right).
\]
in $C^\vee$? In this case, $\chi^{(2,2,1)}+\chi^{(2,1,2)}+\chi^{(2,1,1)}$ generates the ideal. 
Rewriting this in terms of variables corresponding to Hilbert basis elements, this becomes $x_2+x_3+x_4$. 
Hence, $X$ is embedded in $\Aff^4$ with ideal $\langle x_1^2-x_2x_3x_4,x_2+x_3+x_4\rangle$.
\end{ex}

\begin{rem}\label{rem:normal}
If all points $Q_i\in L$ are distinct, the construction of $X(L)$ coincides with that of $X(\D')$, where $\D'$ is the polyhedral divisor 
\[
\D'=\sum_i \Delta_i\cdot Q_i
\]
on $L$. In particular, such $X(L)$ is normal. 

On the other hand, if some points $Q_i$ coincide, $X(L)$ may no longer be normal. We say that the collection of polyhedra $\{\Delta_i\}$ (with tailcone $\sigma$) is \emph{admissible} for $L$ if for every $Q\in L$ either
\begin{enumerate}
\item  $\Delta_i=\emptyset$ for some $Q_i=Q$; or
\item for every $u\in \sigma^\vee$, $\Delta_i(u)$ is non-integral for at most one $i$ with $Q_i=Q$.
\end{enumerate}
Then just as in \cite[Proposition 4.2(iii)]{kstab}, $X(L)$ is normal if and only if $\{\Delta_i\}$ is admissible for $L$. See also \cite[Corollary 2.6]{tdef}.
\end{rem} 
\begin{rem}\label{rem:toric}
If $L$ intersects the boundary $V(y_0y_1\cdots y_m)$ of $\PP^m$ in exactly two points, then $X(L)$ is a (potentially non-normal) toric variety. Indeed, after reordering suppose that $Q_0,\ldots,Q_k$ coincide, as do $Q_{k+1},\ldots,Q_m$. After identifying these points respectively with $0,\infty\in\PP^1$, it follows that $A(L)$ is the semigroup algebra for the semigroup
\[
\left\{ (u,v)\in (\sigma^\vee\cap M)\times \ZZ\ \big |\  - \sum_{i\leq k} \lfloor \Delta_i(u) \rfloor \leq v \leq  \sum_{i\geq k+1} \lfloor \Delta_i(u) \rfloor \right\}.
\]

It is \emph{a priori} not obvious that this is a finitely generated semigroup. However, by our embedding $X(L)$ in $X(C)$, we see not only that the semigroup is finitely generated, but are even given generators for it, namely, the restriction of generators for the semigroup of $C^\vee$.
\end{rem}

\subsection{Similar Constructions}\label{sec:similar}
	Our embedding of $X(L)$ is related to two other constructions. The first is Altmann's construction of so-called toric deformations \cite{altmann:95a}. The setup is as follows: given a pointed polyhedral cone $\sigma\subset N_\QQ$ and an admissible collection $\{\Delta_1,\ldots,\Delta_m\}$ of polyhedra with tailcone $\sigma$ (see Remark \ref{rem:normal}), Altmann constructs an $(m-1)$-parameter deformation of the affine toric variety $X$ whose cone of one-parameter subgroups is generated by $\sigma \times 0$ and $\sum \Delta_i\times 1$ in $N_\QQ\times \QQ$. This is done by embedding $X$ as a relative complete intersection in the affine toric variety $\widetilde X$ whose cone of one-parameter subgroups is generated by $\sigma \times 0$ and $\Delta_i\times e_i$ in $N_\QQ\times \QQ^m$.
	The toric deformation of $X$ arises by perturbing the defining equations of $X$ in $\widetilde{X}$.
It is straightforward to verify that Altmann's embedding coincides with our embedding of $X(L)$ in $X(C)$ when we take as input $\Delta_0=\emptyset$, $\Delta_1,\ldots,\Delta_m$ as above, and $0=Q_1=\ldots=Q_m$, $\infty=Q_0$.

The second related construction is the embedding of a rational complexity-one $T$-variety in a toric variety obtained via their Cox rings. Hausen and S\"u\ss{} describe this embedding in \cite[Corollary 5.2]{hausen} for the case of complete A2 (e.g. projective) varieties, but it also makes sense for affine varieties with no non-constant invariant functions. We briefly recall this embedding and show that it coincides with our embedding of $X$ in $X(C)$.

Fix $X$ a rational complexity-one affine $T$-variety $X$ described by the polyhedral divisor \[\D=\sum_{i=0}^m \Delta_i P_i\] on $\PP^1$ as in \S\ref{sec:tvar}. The condition that $X$ have no non-constant invariant functions corresponds to the requirement that all coefficients of $\D$ are non-empty. Represent each point $P_i\in \PP^1$ by a choice of homogeneous coordinates $(p_{i0}:p_{i1})$.
The \emph{Cox ring} of $X$ has a presentation of the following form, see \cite[Theorem 1.3, Corollary 4.9]{hausen}:
\[
	\cox(X)=\KK[S_\rho,T_{i,v}\ |\ \rho\in\sigma^\times,0\leq i \leq m,v\in\Delta_i^{(0)}]/J
\]
where $\sigma^\times$ is the set of rays of $\sigma$ such that $\rho\cap \sum \Delta_i=\emptyset$, and $\Delta_i^{(0)}$ is the set of vertices of $\Delta_i$. The ideal $J$ is generated by the $3\times 3$ minors of
\[
\left(\begin{array}{c c c c}
	p_{00} & p_{10} & \cdots & p_{m0}\\ 
	p_{01} & p_{11} & \cdots & p_{m1}\\
	T_0 & T_1 & \cdots & T_m
\end{array}
	\right),
\]
where
\begin{align*}
	T_i=\prod_{v\in \Delta_i^{(0)}} 	T_{i,v}^{\mu(v)}
\end{align*}
and $\mu(v)$ denotes the smallest integer $\lambda$ such that $\lambda v\in N_\QQ$ is in $N$.
The $\Cl(X)$-grading on $\cox(X)$ is described by the exact sequence
\[
	\begin{CD}
		0 @>>> M\oplus \ZZ^m @>d>> \ZZ^{\sigma^\times}\bigoplus_i \ZZ^{\Delta_i^{(0)}} @>>> \Cl(X) @>>> 0\\
\end{CD}
\]
where the map $d$ sends
\[
	(u,a)\in M\oplus \ZZ^m\mapsto \sum_\rho  \langle \rho,u\rangle e_\rho + \sum_{i=0}^m\sum_{v\in\Delta_i^{(0)}} \mu(v)(a_i+\langle v,u\rangle)e_{i,v}.
\]
Here, $e_\rho$ and $e_{i,v}$ are the corresponding basis elements, $a_0:=-\sum_{i\neq 0} a_i$, and by abuse of notation we use $\rho$ to denote both a ray and its primitive lattice generator.

Taking the quotient of $\spec \cox(X)$ by the quasitorus $\spec \KK[\Cl(X)]$, we recover $X$. On the other hand, the quotient of the spectrum of the $\Cl(X)$-graded polynomial ring $\KK[S_\rho,T_{i,v}]$ by this quasitorus is an affine toric variety $\widetilde{X}$, see \cite[Chapter 5]{CLS} for details. Since $\spec \cox(X)\subset \spec \KK[S_\rho,T_{i,v}]$, we also obtain a closed embedding $X\subset \widetilde{X}$.

\begin{prop}
The embedding $X\subset \widetilde X$ obtained from the Cox ring of $X$ is the same as the embedding $X\subset X(C)$ described above.
\end{prop}
\begin{proof}
	We begin by considering the Cox ring of $X(C)$, which is a polynomial ring whose free generators correspond to rays of $X(C)$. By construction of $C$, these rays are in bijection with elements of $\sigma^\times$ and elements of $\Delta_i^{(0)}$. Hence, the Cox ring of $X(C)$ is naturally isomorphic to $\KK[S_\rho,T_{i,v}]$ \cite[\S 5.1]{CLS}. In fact, this isomorphism preserves the $\Cl(X)$ grading, since the grading on the Cox ring of $X(C)$ is also induced by the map $d$ above. It follows that $\widetilde {X}$ and $X(C)$ are naturally isomorphic, since both arise as the same quotient of  $\spec \KK[S_\rho,T_{i,v}]$.

	It remains to check that this isomorphism of $\widetilde{X}$ and $X(C)$ maps $X$ to $X$. Since $X$ is integral, we only need to check this on the open torus. For this, it suffices to show that the image of $I(L)$ in $\KK[S_\rho^{\pm 1},T_{i,v}^{\pm 1}]$ agrees with $J$. But the ideal of $L$ in $\PP^m$ is generated by the $3\times 3$ minors of
\[
\left(\begin{array}{c c c c}
	p_{00} & p_{10} & \cdots & p_{m0}\\ 
	p_{01} & p_{11} & \cdots & p_{m1}\\
	y_0 & y_1 & \cdots & y_m
\end{array}
	\right)
\]
for homogeneous coordinates $y_i$ on $\PP^m$, and the inclusion of $\KK [y_i/y_j]\to \KK[S_\rho^{\pm 1},T_{i,v}^{\pm 1}]$ is induced by the map sending $y_i$ to $T_i$. The claim then follows.
\end{proof}

\subsection{Projective $T$-Varieties}
We are also interested in studying embeddings of projective $T$-varieties. Let $X$ be a projective $T$-variety, and $\mcL$ any ample line bundle. Then the ring
\[
R(\mcL)=\bigoplus_{i\in\ZZ_{\geq 0}} H^0(X,\mcL^{\otimes i})
\]
is a finitely generated normal domain, and $\proj R(\mcL)=X$. Choosing homogeneous generators for $R(\mcL)$ as a $\KK$-algebra of degrees $d_0,d_1,\ldots,d_n$ leads to a presentation
\[
\begin{CD}
0 @>>> I @>>> \KK[x_0, \ldots, x_n] @>>> R(\mcL) @>>> 0\\
\end{CD}
\]
and hence an embedding of $X$ in the weighted projective space $\PP(d_0,d_1,\ldots,d_n)$. 

On the other hand, $\spec R(\mcL)$ is an affine $T\times\KK^*$-variety. If we assume that $X$ was rational and of complexity one, then so is $\spec R(\mcL)$, and we have a semi-canonical presentation of $R(\mcL)$ via the semi-canonical embedding $\spec R(\mcL)\subset X(C)$ of \S\ref{sec:affine}. Thus, after choosing the line bundle $\mcL$, we have a semi-canonical embedding of the polarized pair $(X,\mcL)$ in some weighted projective space.

\begin{ex}[The projectivized cotangent bundle on $\PP^2$]\label{ex:proj}
We consider $X=\PP(\Omega_{\PP^2})$, the projectivization of the cotangent bundle on $\PP^2$. This is a Fano threefold often called $W$, and equal to number 2.32 in the list of Mori and Mukai \cite{mori}. It is also isomorphic to the variety of complete flags in $\KK^3$. This variety is equipped with a natural two-torus action.

On $X$, the anticanonical class is divisible by two. We take $\mcL$ to be half the anticanonical bundle and consider the polarised pair $(X,\mcL)$. Using the data found in \cite{picturebook}, one determines that $\spec R(\mcL)$ is encoded by the polyhedral divisor 
\[\D=\Delta_0\cdot\{0\}+\Delta_1\cdot\{1\}+\Delta_2\cdot\{\infty\}\]
on $\PP^1$, where the vertices of $\Delta_0$ are $(1,0,0)$, $(0,0,0)$, the vertices of $\Delta_1$ are $(0,1,0)$, $(0,0,0)$, the vertices of $\Delta_2$ are $(0,0,1)$, $(-1,-1,1)$, and all coefficients have tailcone generated by the columns of 
\[
\left(\begin{array}{c c c c c c}
1 & 1 & 0 & 0 & -1 & -1\\
0 & 1 & 1 & -1 & -1 & 0\\
1 & 1 & 1 & 1 & 1 & 1
\end{array}\right).
\]
The distinguished $\KK^*$-action on $\spec R(\mcL)$ corresponds to the co-character $(0,0,1)$.

The dual cone $C^\vee$ is generated by the columns of 
\[\bgroup\begin{pmatrix}0&
      1&
      0&
      0&
      {-1}&
      {-1}&
      0&
      0&
      1\\
      0&
      0&
      1&
      0&
      0&
      1&
      0&
      {-1}&
      {-1}\\
      1&
      1&
      1&
      1&
      1&
      1&
      1&
      1&
      1\\
      0&
      0&
      0&
      0&
      0&
      0&
      1&
      1&
      1\\
      0&
      0&
      0&
      {-1}&
      {-1}&
      {-1}&
      {-1}&
      {-1}&
      {-1}\\
      \end{pmatrix}\egroup\]
These elements also form a Hilbert basis for the semigroup $C^\vee\cap(M\times\ZZ^m)$. Note that all elements have degree one with respect to our $\ZZ$-grading, so we obtain an embedding of the Fano threefold $X$ in $\PP^8$.

Concretely, the corresponding toric ideal is generated by the $2\times 2$ minors of
\[
\left(\begin{array}{c c c}
x_0&x_4&x_7\\
x_1&x_3&x_8\\
x_2&x_5&x_6
\end{array}\right)
\]
and we recognize that $X(C)$ is simply the (cone over the) Segre embedding of $\PP^2\times \PP^2$. The variety $X$ is further cut out by the additional equation $x_0+x_3+x_6=0$. 
\end{ex}
\begin{rem}
In general, the above construction produces a polarized pair $(X,\mcL)$ in a (potentially weighted) projective space. If we wish to embed $X$ in a standard projective space, we may then pass to the $i$th Veronese subring for appropriate choice of $i$. It is straightforward to check that the construction of \S \ref{sec:affine} commutes with passing to Veronese subrings. 

This means that the construction of \S\ref{sec:affine} applied to $\spec R(\mcL^{\otimes i})$ will yield a semi-canonical embedding of the polarized pair $(X,\mcL^{\otimes i})$ in projective space.
\end{rem}
\begin{rem}\label{rem:cox}
One might also be interested in the total coordinate or \emph{Cox ring} $\cox(X)$ of a rational, complexity-one projective $T$-variety $X$. See \cite{coxrings} for details on Cox rings.

The variety $\spec \cox(X)$ is itself a complexity-one affine $T$-variety, see \cite{pcox} for a description of the corresponding polyhedral divisor. As long as $X$ has only log terminal singularities, $\spec \cox(X)$ will also be rational, see \cite[Corollary 5.11]{iteration}. In these cases, we may apply \S \ref{sec:affine} to produce a semi-canonical presentation for the ring $\cox(X)$. 
See also \S\ref{sec:similar} for further connections with Cox rings.
\end{rem}

\section{The Ideal of $X(L)$}\label{sec:ideal}
In this section, we study the ideal $\widetilde{I}(L)$ of $X(L)$ in the embedding constructed in \S\ref{sec:affine}. In particular, we determine explicitly its generators. This will be important in the subsequent \S\ref{sec:trop}, in which we show that each inital ideal corresponding to a point of $\trop(X(L)^\circ)$ is prime.

We continue using notation as in \S\ref{sec:affine}.
To begin with, we have the following:
\begin{lemma}\label{lemma:preimage}
Let $\pi:\spec \KK[M]\times (\KK^*)^m\to (\KK^*)^m$ be the projection.
Then $\widetilde{I}(L)$ is the ideal of the closure of $\pi^{-1}(L\cap (\KK^*)^m)$ in $X(C)$.
\end{lemma}
\begin{proof}
The ideal $J$ of $\pi^{-1}(L\cap(\KK^*)^m)$ in $\spec \KK[M]\times (\KK^*)^m$ is generated by the image of $I(L)$ in $\KK[M\times\ZZ^m]$. The ideal $J'$ of the closure in $X(C)$ consists of those $g\in J$ such that $g$ is regular on $X(C)$, that is, $g\in A(C)$. Clearly, $\widetilde{I}(L)$ is contained in $J'$.

On the other hand, $g\in J$ is in $A(C)$ if and only if each $M$-graded piece $g_u$ is. But $g_u$ is itself of the form $f\cdot \chi^{u}$ for some $f\in I(L)$, hence $J'$ is contained in $\widetilde{I}(L)$.

\end{proof}
\noindent Geometrically, the above lemma means that the intersection $X^\circ=X\cap(\KK^*)^n$ is just the product of $L^\circ=L\cap(\KK^*)^m$ with the torus $T$.

We are now interested in finding generators for $\widetilde{I}(L)$. 
We set $z_i=y_i/y_0$ for $i=0,\ldots, m$. 
\begin{rem}\label{rem:inac}
Fix a multidegree $u\in M$. 
For any Laurent monomial $z^v\in \KK[z_1^{\pm 1},\ldots,z_m^{\pm 1}]$, $z^v\cdot \chi^u$ is in $A(C)$ if and only if
\begin{align*}
v_i&\geq -\Delta_i(u)\qquad i>0; \ \textrm{and}\\
\sum_{i=1}^m v_i &\leq \Delta_0(u).
\end{align*} 
\end{rem}

Let $\G(L)$ be any set of generators for $I(L)$ which are linear in the $z_i$. 
The other ingredient we need is the polytope $P\subset M_\QQ\times \QQ^m$ consisting of those $(u,v)$ defined by the inequalities
\begin{align*}
v_i&\geq -\Delta_i(u)\qquad i> 0; \ \textrm{and}\\
\sum_{i=1}^m v_i &\leq \Delta_0(u)-1.
\end{align*} 
A set of \emph{lattice generators for $P$ as a $C^\vee$-module} is any set $\mcP\subset P\cap (M\times\ZZ^m)$ such that every lattice point of $\mcP$ is a sum of an element of $\mcP$ with a lattice point of $C^\vee$.

\begin{rem}
 A finite set $\mcP$ of lattice generators for $P$ may be computed by considering a Hilbert basis $\H$ for the cone
\[
\overline{\QQ_{\geq 0}(P\times \{1\})}.
\]
Selecting all elements of $\H$ with final coordinate equal to one, and projecting back to $M\times\ZZ^m$ leads to such a generating set $\mcP$.
\end{rem}

By construction, we know that $\widetilde{I}(L)$ is generated by certain elements of the form $f\cdot \chi^u$, where $f$ is a rational function on $\PP^m$ and $\chi^u$ is a character of $T$. We now specify precisely what form these elements $f$ and $\chi^u$ take:
\begin{prop}\label{prop:ideal}
The ideal $\widetilde{I}(L)$ is generated by the functions $(gz^v)\cdot \chi^{u}$, where $g\in \G(L)$ and $(u,v)\in \mcP$. 
\end{prop}
\begin{proof}
Fix a multidegree $u\in M$. 
It follows from Remark \ref{rem:inac} that the degree $u$ piece of $A(C)$ is 
\begin{align*}
g_u\cdot \chi^{u}\cdot \KK[z_1,\ldots,z_m]_{\leq d_u}\\
\end{align*}
where
\begin{align*}
g_u&=z_1^{-\lfloor \Delta_1(u)\rfloor}\cdots z_m^{-\lfloor \Delta_m(u)\rfloor}\\
d_u&= \sum_i\lfloor \Delta_i(u)\rfloor.\nonumber
\end{align*}

The definition of $P$ is such that $(u,v)\in P$ if and only if 
\[z^v\cdot \chi^u,(z_1 z^v)\cdot\chi^u,\ldots,(z_m z^v)\cdot\chi^u\]
 are all in $A(C)$. In particular, for any $g\in \G(L)$, $g\cdot z^v \cdot\chi^{u}\in \widetilde{I}(L)$ for each $(u,v)\in\mcP$, since $g$ is linear in the $z_i$. 

We now need to show that these elements generate all of $\widetilde{I}(L)$.
Consider any homogeneous element in $\widetilde{I}(L)\subset A(C)$, which (by the above) we may write as $f g_u\cdot  \chi^u$, with $f\in \KK[z_1,\ldots,z_m]_{\leq d_u}$.
The polynomial $f$ is in $I(L)\cap \KK[z_1,\ldots,z_m]$. Using that the $g\in \G(L)$ are linear and generate $I(L)$, we may write 
\[
f=\sum_{g\in\G(L)} c_g g
\]
with all $c_g g \in\KK[z_1,\ldots,z_m]_{\leq d_u}$.

Hence, we have reduced to showing that we can generate any element of the form $g  z^v\cdot \chi^{u}$ such that $g z^v\cdot  \chi^{u}\in A(C)$, that is, $(u,v)\in P$. Now, any $(u,v)\in P$ can be written as $(u',v')+(u'',v'')$, where $(u',v')\in \mcP$ and $(u'',v'')\in C^\vee$. Thus, we can generate $g z^v\cdot \chi^{u}$ as
\[
g z^v\cdot\chi^{u}=(z^{v''}\cdot\chi^{u''})(g  z^{v'}\cdot \chi^{u'}),
\]
noting that $z^{v''}\cdot \chi^{u''}$ is a regular function.
\end{proof}

\section{Tropicalization and Gr\"obner Theory}\label{sec:groebner}

\subsection{Basics}\label{sec:tropical}

Here we review the elements of Gr\"obner theory and tropical geometry necessary for our results.  We recommend the books \cite{CLO,Sturmfels-GBCP,tropical} as references.  We will be working in the polynomial ring $\KK[\bx]=\KK[x_1, \ldots, x_n]$ and the corresponding ring of Laurent polynomials $\KK[\bx^{\pm1}]=\KK[x_1^{\pm 1},\ldots,x_n^{\pm 1}]$.

Let $f = \sum c_{\alpha}x^{\alpha} \in \KK[\bx^{\pm1}]$ and $w \in \RR^n$. The \emph{initial form} $\In_w(f)$ of $f$ with respect to $w\in\RR^n$ is
\[
\In_w(f)=\sum_{\substack{\alpha\\ \alpha\cdot w\ \textrm{minimal}}} c_\alpha x^\alpha.
\]
For an ideal $I$ in $\KK[\bx]$ or $\KK[\bx^{\pm 1}]$, the \emph{initial ideal} $\In_w(I)$ is the ideal generated by the set $\{\In_w(f) \mid f \in I\}$.

\begin{defn}
Let $X\subset \Aff^n$ be an affine variety intersecting the torus with ideal $I$, and $I^\circ$ the ideal of $X^\circ=X\cap(\KK^*)^n$. The \emph{tropicalization} $\trop(X^\circ)$ of $X^\circ$ is the set of those $w\in \RR^n$ such that $\In_w(I^\circ)$ does not contain a monomial.
\end{defn}

\begin{rem}
It is straightforward to check that $\In_w(I^\circ)$ does not contain a monomial if and only if the same holds true for $\In_w(I)$.
\end{rem}

The tropical variety $\trop(X^\circ)$ is typically given the structure of a polyhedral fan with the following property: if two elements $w, w' \in \trop(X^\circ)$ are in the interior of the same face, then $\In_w(I^\circ)=\In_{w'}(I^\circ)$. The structure theorem for tropical varieties \cite[Theorem 3.3.6]{tropical} says that $\trop(X^\circ)$ has dimension $d=\dim X$ and is connected in codimension $1$.  This means that any two maximal dimensional faces $C, C' \subset \trop(X^\circ)$ can be connected by a path which is contained in the relative interiors of $d$ and $(d-1)$-dimensional cones.

Since we are interested in degenerations not of $X^\circ$ but of $X$, we consider a polyhedral structure so that if $w, w' \in \trop(X^\circ) \subset \RR^n$ belong to the interior of the same cone then $\In_w(I) = \In_{w'}(I)$.
This can be taken as a refinement of the above polyhedral structure, since it is straightforward to check that 
 $\In_w(I)=\In_{w'}(I)$ implies $\In_w(I^\circ)=\In_{w'}(I^\circ)$.
 Note that for homogeneous ideals $I$, this polyhedral structure on $\trop(X^\circ)$ is exactly that induced by the so-called \emph{Gr\"obner fan}. For non-homogeneous ideals, one may consider the homogenization of $I$ in $\KK[x_0,\ldots,x_n]$. Intersecting the Gr\"obner fan for this ideal in $\RR^{n+1}$ with the $0$th coordinate hyperplane yields a fan $\Sigma$ in $\RR^n$ inducing the desired polyhedral structure on $\trop(X^\circ)$. Indeed, \cite[Proposition 3.2.8]{tropical} shows $\trop(X^\circ)$ has the structure of a closed subfan of $\Sigma$. Furthermore, \cite[Proposition 2.6.1]{tropical} shows that any two elements in the interior of a common cone of $\Sigma$ induce the same initial ideal of $I$.

 \begin{defn}
We call a cone of $\trop(X^\circ)$ \emph{prime} if the initial ideal corresponding to an interior weight vector is prime.
\end{defn}
\noindent It follows that the embedding of $X$ is well-poised exactly when every cone of $\trop(X^\circ)$ is a prime cone.

In order to effectively compute the tropicalization of a variety, one desires a (finite) {tropical basis}:
\begin{defn}
A \emph{tropical basis} for $X^\circ$ is a  subset $G \subset I^\circ$ such that \[\trop(X^\circ) = \bigcap_{f \in G} \trop(V(f)).\]
\end{defn}
\noindent 
For any $G\subset I^\circ$, $\trop(X^\circ)$ is always contained in $\bigcap_{f \in G} \trop(V(f))$, but even if $G$ generates $I^\circ$, equality might not hold. However, every  $X^\circ\subset(\KK^*)^n$ does possess a finite tropical basis, see \cite[Theorem 2.6.5]{tropical}.

\subsection{Tropicalization of Semi-Canonical Embeddings}\label{sec:trop}
We now assume that we are in the situation of \S \ref{sec:affine}: we have fixed $m+1$ polyhedra $\Delta_i$ with pointed tailcone $\sigma$, along with a line $L\subset \PP^m$ intersecting the dense torus. By Theorem \ref{thm:embed}, we have an embedding of the affine variety $X(L)$ in the toric variety $X(C)$ with ideal $\widetilde{I}(L)$. 

Fixing a generating set $H$ of $C^\vee\cap (M\times\ZZ^m)$ of size $n$ leads to an embedding of $X(C)$ and thus $X(L)$ into $\Aff^{n}$.  We consider the polynomial ring
\[
S=\KK[x_u\ |\ u\in H ].
\]
Let $J(L)$ be the ideal in $S$ of $X(L)$ under this embedding. We call the images of the $x_u$ for $u\in H$ the \emph{semi-canonical generators} of $A(L)$. 

We are interested in describing $\trop(X(L)^\circ)$, where $X(L)^\circ=X(L)\cap(\KK^*)^n$.
\begin{prop}\label{prop:tropmap}
The tropicalization $\trop(X(L)^\circ)$ is the image of $N_\RR\times\trop(L^\circ)$ in $\RR^{\#H}$ under the injective linear map
\begin{align*}
\phi:(N\oplus\ZZ^m)_\RR &\to \RR^{\#H}\\
v&\mapsto (\ldots,\langle v,u\rangle,\ldots)_{u\in H}.
\end{align*}
In particular, $\trop(X(L)^\circ)$ is the product of a tropical line with a plane of dimension $(\rank N)$.
\end{prop}
\begin{proof}
By construction, $X(L)\subset X(C)$, and $X(L)^\circ$ is the image in $\Aff^{n}$ of $U:=X(L)\cap \spec \KK[M\times \ZZ^m]$. 
By Lemma \ref{lemma:preimage}, $U$ is just $T\times L^\circ$, and the map $X(C)\to \Aff^n$ is given coordinate-wise by the functions $\chi^u$ for $u\in H$. Since the tropicalization of this monomial map is exactly $\phi$, applying the tropicalization functor \cite[Corollary 3.2.13]{tropical} yields the claim.
\end{proof}

If the intersection points $Q_i$ of $L$ with the coordinate hyperplanes are all distinct, $\trop(L^\circ)$ is the one-dimensional fan in $\RR^m$ with rays $\rho_0,\rho_1,\ldots,\rho_m$ generated by $e_0,e_1,\ldots,e_m$. This is the case if $L=L_\D$ as in Remark \ref{rem:choice}.
 It follows that $\trop(X(L)^\circ)$ has $m+1$ maximal dimensional cones $C_0,\ldots,C_m$, with $C_i=\phi(\rho_i\times N_\RR)$. 
 We picture $\trop (L^\circ)$ in Figure \ref{fig:line} in the case $m=2$.
If some of the intersection points $Q_i$ coincide, then $\trop(L^\circ)$ is a degenerate tropical line, and has fewer than $m+1$ rays.

 \begin{figure}
	\begin{tikzpicture}[scale=.6]
    \draw [-] (0,0) -- (-2,-2);
    \draw [-] (0,0) -- (2,0);
    \draw [-] (0,0) -- (0,2);
    \end{tikzpicture}
 \caption{$\trop (L^\circ)$ in the case $m=2$}\label{fig:line}
 \end{figure}
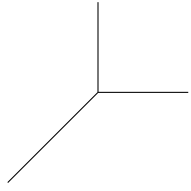

\begin{cor}\label{cor:tropbasis}
Let $\G(L)$ be a tropical basis for $L^\circ$ whose elements are dehomogenizations of linear forms. Then a tropical basis for $X(L)^\circ$ which generates the ideal $J(L)\subset S$ may be constructed by taking generators of the toric ideal of $X(C)\subset \Aff^n$, along with lifts to $S$ of the generators of $\widetilde{I}(L)$ produced by Proposition \ref{prop:ideal}.
\end{cor}
\begin{proof}
After tropicalization, the generators of the toric ideal cut out the linear space given by the image of $\phi$. It remains to cut out $\trop(X(L)^\circ)$ inside this linear subspace. The generators of $\widetilde{I}(L)$ produced by Proposition \ref{prop:ideal} cut out $N_\RR\times \trop(L^\circ)$, so their lifts to $S$ will cut out $\trop(X(L)^\circ)$ after intersecting with $\im \phi$.  
\end{proof}

We also wish to describe $\In_w(J(L))$ for $w\in \trop(X(L)^\circ)$. 
Since $\trop(X(L)^\circ)$ is contained in the image of the injective linear map $\phi$ above,  weights $w\in\trop(X(L)^\circ)$ in this subspace can be interpreted as giving weights on monomials of $\KK[M\times\ZZ^m]$. In particular, any $w\in \trop(X(L)^\circ)$ can be interpreted as a weight on $\KK[\ZZ^m]$ or on the monomials of $A(C)$.
\begin{defn}
For any $w\in \trop(X(L)^\circ)$, 
\[
L_w:=\overline {V(\In_w I(L))}\subset\PP^m.
\]
\end{defn}
\begin{lemma}The variety $L_w$ is a line in $\PP^m$ intersecting the torus.
\end{lemma}
\begin{proof}
Clearly the variety $L_w$ is a line in $\PP^m$. We must thus show that $\In_w I(L)$ does not contain a monomial.
If instead  $\In_w I(L)$ contains a monomial, there is some $f\in I(L)$ such that $\In_w f$ is a monomial. Furthermore, there is some degree $u\in M$ such that $f\cdot \chi^u\in \widetilde{I}(L)$, hence $\In_w \widetilde{I}(L)$ and thus $\In_w J(L)$ also contain a monomial. But this contradicts $w\in\trop(X(L)^\circ)$.
\end{proof}
Given the above lemma, we may apply the construction of \S\ref{sec:affine} to obtain an irreducible variety $X(L_w)$ embedded in $\Aff^{n}$ via an embedding in $X(C)$. Let $J(L_w)$ denote the ideal of this affine variety.  

\begin{thm}\label{thm:poised}
For all $w\in \trop(X(L)^\circ)$, \[\In_w J(L)=J(L_w).\] In particular, the embedding $X(L)\hookrightarrow \Aff^{n}$ is well-poised.
\end{thm}

\begin{proof}
Note that $J(L)$ is just the preimage of $\widetilde{I}(L)$ in $S$. Likewise, $J(L_w)$ is the preimage of $\widetilde{I}(L_w)$ in $S$. Hence, it will suffice to show that 
\[
\widetilde{I}(L_w)=\In_w(\widetilde{I}(L))
\]
with $w$ interpreted as a weight in $A(C)$.

To that end, consider a homogeneous element $f\cdot \chi^u\in\widetilde{I}(L)$, where $u\in M$ and $f\in I(L)$. Then 
\[
\In_w(f\cdot \chi^u)=\In_w(f)\cdot \chi^u.
\]
Since $\In_w(f)\in \In_w(I(L))$, we obtain that $\In_w(\widetilde{I}(L))\subset \widetilde{I}(L_w)$.

The reverse inclusion follows from Proposition \ref{prop:ideal}. Indeed, let $\G(L)$ be any set of generators of $I(L)$ linear in the $z_i$, whose initial terms with respect to $w$ generate $\In_w(I(L))=I(L_w)$. By Proposition \ref{prop:ideal}, $\In_w(g) z^v\cdot \chi^u$ generate $\widetilde{I}(L_w)$ for $g\in \G(L)$ and $(v,w)\in\mcP$. But by the same proposition, $g z^v\cdot \chi^u$ is then in $\widetilde{I}(L)$. Hence, $\In_w(\widetilde{I}(L))$ generates $\widetilde{I}(L_w)$.
\end{proof}

\begin{ex}\label{ex:d62}
We continue Example \ref{ex:d6}.
The tropicalization of $X$ has three maximal cones. The corresponding initial ideals are (up to permutation of indices) all generated by $x_1^2+x_2^2x_3$. This corresponds to a degeneration to a non-normal toric variety.

Note that this embedding of $X(L)$ is not minimal: we can eliminate e.g. $x_4$ and arrive at the single equation $x_1^2+x_2^2x_3+x_2x_3^2=0$. This minimal embedding of $X(L)$ is not well-poised: one of the initial ideals is generated by $x_2^2x_3+x_2x_3^3$, which is not prime.
\end{ex}

\begin{ex}\label{ex:proj2}
  We continue Example \ref{ex:proj}
  As in Example \ref{ex:d62}, the tropicalization has three maximal cones. The corresponding initial ideals are (up to permutation of indices) all generated by the $2\times 2$ minors of 
\[
\left(\begin{array}{c c c}
x_0&x_4&x_7\\
x_1&-x_0&x_8\\
x_2&x_5&x_6
\end{array}\right)
\]
One may check directly that this initial ideal defines a normal toric variety. Alternatively, since the coefficients $\Delta_0,\Delta_1,\Delta_2$ are all lattice polytopes, Remarks \ref{rem:normal} and \ref{rem:toric} immediately imply that this is the case. 
  \end{ex}

  \section{Valuations and Newton-Okounkov Bodies}\label{sec:part2}
\subsection{Basics}\label{sec:val}
Let $R$ be a $\KK$-domain of finite type of Krull dimension $d$, and $(\Gamma,<)$ a totally ordered, finitely generated free abelian group. A \emph{$\KK$-valuation} of $R$ with values in $\Gamma$ is a map $\fv: R \setminus \{0\} \to \Gamma$ such that  $\fv(\KK^*)=0$, and for all $f,g\in R$, $\fv(fg)=\fv(f)+\fv(g)$ and $\fv(f+g)\geq \min \{\fv(f),\fv(g)\}$.   
For details on valuations, see \cite[Chapter VI]{ZS}.
The image of $\fv$ is denoted $S(R, \fv)$; it is a semigroup in $\Gamma$ known as the \emph{value semigroup}. We denote the closure of the positive hull of $S(R,\fv)$ by $C(R,\fv)$.

Any valuation $\fv$ on $R$ can be extended uniquely to a valuation on the field of fractions of $R$. 
Following standard terminology, the \emph{rank} $r(\fv)$  of $\fv$ is the rank of the free Abelian group generated by $S(R, \fv)$. This latter group is the same as the image of the extension of $\fv$ to the field of fractions of $R$. The rank $r(\fv)$ is always at most the Krull dimension of $R$, and we say that $\fv$ is of \emph{full rank} if equality holds.  For any full rank valuation $\fv$, a \emph{Khovanskii basis} for $\fv$ is a subset of $R$ whose image under $\fv$ generates $S(R,\fv)$, see e.g.~\cite{Kaveh-Manon-NOK}.\footnote{Although it appears different, our definition of Khovanskii basis agrees with that of \cite{Kaveh-Manon-NOK}, see \cite[Proposition 2.4]{Kaveh-Manon-NOK}.}

\subsection{Homogeneous Valuations on Complexity-One $T$-Varieties}\label{sec:homog}
Suppose that the domain $R$ is graded by a lattice $M$. We say that a valuation $\fv$ is \emph{$M$-homogeneous} (or just homogeneous) if for all $f\in R$,
\[
	\fv(f)=\min_{u\in M} \{\fv(f_u)\}
\]
where $f=\sum_{u\in M} f_u$ is the decomposition of $f$ into graded pieces.

We will now study homogeneous valuations (and their value semigroups) for coordinate rings of affine rational complexity-one $T$-varieties. Indeed, let $X=X(L)$ be an affine rational complexity-one $T$-variety as in \S\ref{sec:affine}. The coordinate ring $R_X=A(L)$ is $M$-graded, and localizes to the ring $\KK[L^\circ][M]$.

Fix a totally ordered finitely generated abelian group $\Gamma$, a homomorphism $\psi:M\to \Gamma$, a point $Q\in L$, and an element $\gamma\in \Gamma$, $\gamma\geq 0$. The tuple $(\psi, Q, \gamma)$ determines a homogeneous valuation $\fv_{(\psi,Q,\gamma)}$ of $R_X$ by 
\[
f\cdot \chi^u\mapsto \psi(u)+\ord_Q(f)\cdot \gamma
\]
for $f\in \KK(L)$, $u\in M$. Here, $\ord_Q(f)$ is the order of vanishing of $f$ at $Q$. It is straightforward to check that this indeed determines a valuation.

\begin{prop}\label{prop:homog}
Every $M$-homogeneous valuation on $R_X$ is of the form $\fv_{(\psi,Q,\gamma)}$.
\end{prop}
\begin{proof}
	The data of a valuation on $R_X$ is equivalent to the data of a valuation on $\KK(L)[M]$. Homogeneous valuations on the latter are determined by their restrictions to $\KK(L)$ and $\KK[M]$. A homogeneous valuation on $\KK[M]$ is simply a homomorphism $\psi:M\to \Gamma$. Finally, every valuation of $\KK(L)$ is of the form $f\mapsto \ord_Q(f)\cdot \gamma$ for some $Q\in L$ and some $\gamma\in\Gamma$, $\gamma\geq 0$, see Lemma \ref{lemma:val} below. 
\end{proof}

The following lemma is well-known; we include it and its proof for lack of a suitable reference.
\begin{lemma}\label{lemma:val}
Every $\KK$-valuation $\fv$ on $\KK(\PP^1)$ is of the form $f\mapsto \ord_Q(f)\cdot \gamma$ for some $Q\in \PP^1$ and some $\gamma\in\Gamma$, $\gamma\geq 0$.
\end{lemma}
\begin{proof}
We may assume that $\fv$ is non-constant, otherwise we just take $\gamma=0$.
Say that $\KK(\PP^1)\cong \KK(t)$ and let $S$ be the valuation ring of $\fv$. By potentially replacing $t$ with its inverse, we have $\KK[t]\subseteq S$. The intersection of the maximal ideal of $S$ with $\KK[t]$ is a prime ideal of the form $P=\langle t-c \rangle$ for some $c\in \KK$. Then clearly $\KK[t]_P\subseteq S$, but the reverse inclusion holds as well, since otherwise $(t-c)^{-1}\in S$, implying $S=\KK(t)$ which contradicts that $\fv$ is non-constant. It follows that $\fv$ is given by $f\mapsto \ord_Q(f)\cdot \gamma$ where $Q$ is the point corresponding to $t-c$ and $\gamma$ is the valuation of $t-c$.
\end{proof}

\begin{rem}
The valuation $\fv_{(\psi,Q,\gamma)}$ is of full rank if and only if $\psi$ is of full rank and no multiple of $\gamma$ is contained in the image of $\psi$. 
\end{rem}

\begin{thm}\label{thm:khovanskii}
Let $\fv=\fv_{(\psi,Q,\gamma)}$ for any $Q\in L\setminus L^\circ$. Then the valuations of the semi-canonical generators of $A(L)$ generate the  value semi-group $S(R_X,\fv)$. In particular, they form a Khovanskii basis if $\fv$ is of full rank.
\end{thm}
\begin{proof}
Consider any homogeneous element $f\cdot \chi^u\in A(L)$. The claims will follow if we can find a monomial in the semi-canonical generators whose valuation is the same as that of $f\cdot \chi^u$. 

Without loss of generality, assume that $Q=Q_1$. 
Now, $f\cdot \chi^u\in A(L)$ implies that 
\[
	v_i:=\ord_{Q_i} f \geq -\Delta_i(u)
\]
and
\[
	\sum_{i=1}^m v_i\leq \Delta_0(u).
\]
But the monomial $g=z^v\cdot \chi^u$ is then in $A(C)$, see Remark \ref{rem:inac}. This monomial is the image in $A(C)$ of a monomial in the semi-canonical generators. Furthermore if $\overline g$ is the image of $g$ in $A(L)$, we have
\[
	\fv(\overline g)=\fv(\overline{z^v})+\fv(\chi^u)=\fv(z_1^{v_1})+\fv(\chi^u)=\fv(f)+\fv(\chi^u)=\fv(f\cdot\chi^u).
\]

\end{proof}

\begin{cor}\label{cor:cone}
With $\fv$ as above and $Q=Q_j$, the value semigroup $S(R_X,\fv)$ is the image of 
\begin{equation}\label{eqn:sg}
\left\{ (u,v)\in (\sigma^\vee\cap M)\times \ZZ\ \big |\  -  \lfloor \Delta_j(u) \rfloor
\leq v \leq  \sum_{i\neq j} \lfloor \Delta_i(u) \rfloor \right\}
\end{equation}
under the map $\rho:M\times \ZZ \to \Gamma$ given by $(u,\lambda)\mapsto \psi(u)+\lambda\cdot \gamma$. In particular, if $\fv$ has full rank, then $S(R,\fv)$ is isomorphic to the semigroup of \eqref{eqn:sg}.
\end{cor}
\begin{proof}
	It is straightforward to check that \eqref{eqn:sg} is exactly the image of the monomials in the semi-canonical generators under the valuation $\fv_{(\id_M,Q_j,(0,1))}$. The claim follows from the theorem, since 
	\[
\fv=\rho\circ \fv_{(\id_M,Q_j,(0,1))}.
	\]
\end{proof}

\begin{rem}
	The restriction in Theorem \ref{thm:khovanskii} and Corollary \ref{cor:cone} that $Q\in L\setminus L^\circ$ is not a significant one: for any point $Q\in L$, one can re-embed $L$ in a larger projective space so that the points of this new line outside the torus contain the original points $Q_0,\ldots,Q_m$ of $L$ together with $Q$. In particular, for any full rank homogeneous valuation on $R_X$, there is a semi-canonical embedding such that the semi-canonical generators are a Khovanskii basis.
\end{rem}

\begin{ex}
We continue Examples \ref{ex:d6} and \ref{ex:d62} for $X$ the $D_6$ singularity. Consider the valuation $\fv=\fv_{(\psi,Q_0,\gamma)}$ with $\psi:\ZZ\to \Gamma=\ZZ^2$ the inclusion of the first factor, and $\gamma$ the second standard basis vector of $\ZZ^2$. Then by Corollary \ref{cor:cone}, the value semigroup $S(R_X,\fv)$ equals 
\begin{equation*}
\left\{ (u,v)\in \ZZ_{\geq 0}\times \ZZ\ \big |\    \frac{-3u}{2} 
\leq v \leq  2\left\lfloor \frac{-u}{2} \right\rfloor \right\}.
\end{equation*}
Generators are given by the valuations of the images of
\[\chi^{3}z_1^2z_2^2,\ \chi^{2}z_1^1z_2^2,\ \chi^{2}z_1^2,z_2^1,\ \chi^{2}z_1^1z_2^1,\] which are respectively $(3,-4)$, $(2,-3)$, $(2,-3)$, and $(2,-2)$, see Figure \ref{fig:d6}.
  \end{ex}
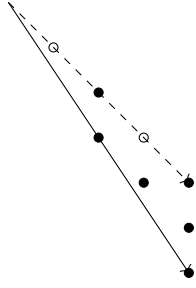
\begin{figure}
  \begin{tikzpicture}[scale=.6]
    \draw [->] (0,0) -- (4,-6);
    \draw [dashed,->] (0,0) -- (4,-4);
    \draw (1,-1) circle [radius=0.1];
    \draw (3,-3) circle [radius=0.1];
    \draw[fill] (2,-2) circle [radius=0.1];
    \draw[fill] (2,-3) circle [radius=0.1];
    \draw[fill] (3,-4) circle [radius=0.1];
    \draw[fill] (4,-4) circle [radius=0.1];
    \draw[fill] (4,-5) circle [radius=0.1];
        \draw[fill] (4,-6) circle [radius=0.1];
    \end{tikzpicture}
\caption{Value semigroup for $D_6$ singularity}\label{fig:d6}
\end{figure}

\subsection{Connections to Tropical Geometry}\label{sec:tropval}

Let $X\subset \Aff^n$ be an affine variety, and fix a presentation of its coordinate ring:

$$
\begin{CD}
0 @>>> I_X @>>> \KK[\bx] @>\pi>> R_X @>>> 0.\\
\end{CD}
$$

In \cite[\S 3,4]{Kaveh-Manon-NOK}, Kaveh and the second author describe certain \emph{weight valuations} $\fv_W: R_X \setminus \{0\} \to \ZZ^r$ corresponding to a special choice of matrix $W$. The $n$ columns of $W$ are chosen in bijection with the variables $\bx$ from the group $\ZZ^r$ taken with its standard lexicographic ordering.  This choice defines a natural $\ZZ^r$-weighting of the monomials in $\KK[\bx]$, in particular $\bx^{\alpha}$ is assigned $W\alpha \in \ZZ^r$.   The value $\fv_W(f)$ for $f \in R_X$ is then computed by the following formula:

\begin{equation*}
\fv_W(f) = \max\{\min\{ W\alpha \ \mid \ p(\bx) = \sum C_{\alpha}\bx^{\alpha},  C_{\alpha} \neq 0\} \ \mid \ p(\bx) \in\pi^{-1}(f)\}.
\end{equation*}

One may also associate an initial ideal to the matrix $W$:
\[
\In_W(I_X)=\langle \In_W(f)\ |\ f\in I\rangle
\]
where for $f=\sum c_\alpha \bx^\alpha$, $\In_W(f)$ is the sum of monomial terms $c_\alpha\bx^\alpha$ for which $W\alpha$ is minimal.
This leads to the notion of \emph{higher rank} tropical varieties:
\[
\trop^r(I_X)=\{W\in \QQ^{r\times n}\ |\ \In_W(I_X)\ \textrm{contains no monomials}\}.
\]
Note that our definition of $\In_W(I_X)$ extends naturally to matrices $W$ with entries in $\QQ$.

If $R_1, \ldots, R_r$ are the rows of $W$, \cite[Lemma 3.8 and Proposition 3.16]{Kaveh-Manon-NOK} imply that:
\begin{enumerate}
\item $\In_W(I_X) = \In_{R_r}(\ldots \In_{R_1}(I_X)\ldots )$,
\item $W\in\trop^r(I_X)$ if and only if $R_i \in \trop(\In_{R_{i-1}}(\ldots \In_{R_1}(I_X)\ldots ))$ for all $i$.
\end{enumerate}

In general, the function $\fv_W$ is only a quasi-valuation (see \cite[\S 4]{Kaveh-Manon-NOK}): it only holds that $\fv_W(fg)\geq \fv_W(f)+\fv_W(g)$ instead of the usual equality.
However, $\fv_W$ defines a valuation of rank equal to the rank of $W$ provided the associated initial ideal $\In_W(I_X)$ is a prime ideal. Indeed, \cite[Lemma 4.4]{Kaveh-Manon-NOK} tells us that in this situation, the associated graded algebra to $\fv_W$ is a domain, which in turn implies that $\fv_W$ is a valuation.  Under mild conditions, the columns of $W$ are equal to the values of $\fv_W$ on the chosen generators of $R_X$; this is in particular the case if $W$ is a member of the higher rank tropical variety $\trop^r(I_X)$, see \cite[Proposition 4.6(4)]{Kaveh-Manon-NOK}.

If $W\in\trop^r(I_X)$ and $\In_W(I_X)$ is prime, then it follows from the above that the initial ideal $\In_W(I_X)$ coincides with the initial ideal of a prime cone in $\trop(X^\circ)$.  As the following theorem demonstrates, the situation is somewhat simplified for affine complexity-one $T$-varieties and their semi-canonical embeddings.

\begin{thm}\label{thm:chris}
	Let $X$ be a rational affine complexity-one $T$-variety, and $\fv: R_X \to \ZZ^r$ a full rank valuation. The following are equivalent: 

\begin{enumerate}
\item The valuation $\fv$ is $M$-homogeneous;
\item There exists a semi-canonical embedding of $X$ such that the semi-canonical generating set is a Khovanskii basis for $\fv$;
\item The valuation $\fv$ equals $\fv_W$ for some $W \in \trop^r(J(L))$ with $rank(W) = r$, where $J(L)$ is the ideal of forms vanishing on a semi-canonical generating set of $R_X=A(L)$.
\end{enumerate}

\noindent
In the case that the above conditions hold, $\fv|_{\KK(L)} = \ord_Q\cdot \gamma$ for some $\gamma \in \ZZ^r$ and some $Q=Q_i \in L \setminus L^\circ$. The initial ideal $\In_W(J(L))$ coincides with the initial ideal of the facet $C_i \subset \trop(J(L))$.  Furthermore, the value semigroup $S(R_X, \fv)$ is generated by the columns of $W$.  
\end{thm}

\begin{proof}
${\bf (1 \to 2)}$ According to Proposition \ref{prop:homog},  $\fv|_{\KK(L)} = ord_Q\cdot \gamma$ for some $\gamma \in \Gamma$. We can re-embed $L$ in a projective space so that the boundary points $L \setminus L^\circ$ contain $Q$. Theorem \ref{thm:khovanskii} completes the claim.  

${\bf (2 \to 3)}$  If a semi-canonical generating set $\mathcal{B} \subset A(L)$ is a Khovanskii basis, \cite[Theorem 2.16]{Kaveh-Manon-NOK} implies that $\fv$ is a \emph{subductive valuation}. Then \cite[Lemma 4.10]{Kaveh-Manon-NOK} implies that $\fv = \fv_W$ for $W$ a weight matrix of rank equal to the rank of $\fv$. In fact, the columns of $W$ are just given by $\fv(b)$, as $b$ ranges over the elements of $\mathcal{B}$. 
Furthermore, $W\in \trop^r(J(L))$, where $J(L)$ is the ideal of forms vanishing on the given semi-canonical generating set. Indeed, given any $f=\sum c_\alpha x^\alpha \in J(L)$, we have 
\[0=\sum c_\alpha \pi(x^\alpha),\] where $\pi:\KK[x]\to R_X$ is the projection induced by the elements of $\mathcal{B}$. Then $\fv$ must obtain its minimum twice on the terms in the right hand side of this expression, so $W$ must obtain its minimum twice on the exponents occuring in $f$. Hence $\In_W(f)$ cannot be a monomial.

${\bf (3 \to 1)}$ Any element of a semi-canonical generating set is of the form $f\chi^u$ for $f \in \KK(L)$ and $u \in M$.  We may thus write $W = W_1 + W_2$, where the columns of $W_1$ are the values $\fv_W(\chi^u)$ for $\chi^u \in \KK[M]$ and the columns of $W_2$ are of the form $s\cdot \gamma$ for some fixed $\gamma \in \ZZ^r$ and $s = \ord_Q(f)$ for $Q \in L \setminus L^\circ$.  It follows that $W_1$ must have rank $r-1$ and $\gamma$ must be independent of the span of $W_1$.  Considering any  $u\neq  v \in M$, the above description of $W$ implies that the values of two homogeneous elements $f\chi^u, g\chi^{v}$ are distinct; it follows that $\fv_W$ is $M$-homogeneous.

  We now show the remaining claims. Note first that we must have $\In_W(J(L)) = \In_{R_r}(\ldots \In_{R_1}(J(L))\ldots )$.  Let $j$ be the index of the first non-zero entry of $\gamma$. Then the row $R_j$ is of the form $\phi(w+s(\ord_Q(y_1), \ldots, \ord_Q(y_m)))$ for some $s\neq 0$ and $w\in N$, where $\phi$ is the map from Proposition \ref{prop:tropmap}.  It follows that $Q=Q_i$ for some $i$, otherwise $\ord_Q(y_i)=0$ for all $i$. Then  $R_j \in C_i \subset \trop(J(L))$.  Furthermore, each row $R_{j'}$ with $j' < j$ is in the lineality space of $J(L)$, and each row $R_{j'}$ with ${j'} > j$ is in the lineality space of $\In_{R_j}(J(L))$; it follows that $\In_W(J(L)) = \In_{R_j}(J(L))$.  
\end{proof}

\subsection{Graded Rings and Newton-Okounkov Bodies}\label{sec:flag}
Suppose that our finitely-generated $\KK$-domain $R$ is graded by $\ZZ_{\geq 0}$. Consider a homogeneous valuation $\fv:R\setminus\{0\}\to \ZZ\times \Gamma$ such that for any homogeneous $f\in R$ of degree $k$, the first coordinate of $\fv(f)$ is exactly $k\in\ZZ$. Then the \emph{Newton-Okounkov Body} of $R$ with respect to $\fv$ is
\[
\Delta(R,\fv)=\pi_2(C(R,\fv)\cap \pi_1^{-1}(1))
\]
where $\pi_1,\pi_2$ are the projections of $(\ZZ\times \Gamma)_\RR$ to $\ZZ_\RR$ and $\Gamma_\RR$. This is a convex set, and under some mild hypotheses (including that $R_0=\KK$), it is compact, see e.g. \cite{KK}.

Such a situation occurs when considering a divisor $D$ on a projective variety $X$. Given any $\Gamma$-valued valuation $\fv$ on $\KK(X)$, one constructs a homogeneous valuation $\fv_D$ on  
\[
R(\CO(D))=\bigoplus_{k\in \ZZ_{\geq 0}} H^0(X,\CO(kD))
\]
as follows: for a section $s\in H^0(X,\CO(kD))$ we have
\[
\fv_{D}(s)=(k,\fv(s))\in\ZZ\times\Gamma.
\]
The Newton-Okounkov body of $D$ with respect to $\fv$ is $\Delta(R(\CO(D),\fv_D)$.

For $X$ a $d$-dimensional projective variety, a distinguished class of full rank valuations on $\KK(X)$ comes from a choice of full flag $\F_0\subset\F_1\subset\ldots\subset\F_d=X$ of irreducible subvarieties of $X$,
 where the point $\mathcal{F}_0$ is a smooth point in each $\mathcal{F}_i$  
(see \cite[\S 1.1]{LM}, \cite[Example 2.13]{KK}).
  For $f \in \KK(X)$ one computes the value $\fv_{\mathcal{F}}(f) = (a_1, \ldots, a_d) \in \ZZ^d$ recursively, where  $\ZZ^d$ is endowed with the lexicographic ordering.  The first component $a_1$ is taken to be the order of vanishing of $f$ along the divisor $\mathcal{F}_{d-1}$.  It follows that if $s$ is a local equation for $\mathcal{F}_{d-1}$ at $p$, $s^{-a_1}f$ can be regarded as a non-zero rational function on $\mathcal{F}_{d-1}$. This allows the procedure to be repeated with the divisor $\mathcal{F}_{d-2} \subset \mathcal{F}_{d-1}$ to produce $a_2$, and so on until the process terminates with $a_d$, the order of vanishing at $\F_0$.   

In the case of such a valuation $\fv_\F$, it is customary to extend the valuation to a homogeneous valuation $\fv_{\F,D}'$ on  $R(\CO(D))$ as follows: if $f_D$ is a local equation for $D$ at the point $\F_0$, for a section $s\in H^0(X,\CO(kD))$ we have
\[
\fv_{\F,D}'(s)=(k,\fv_\F(s)+k\fv_\F(f_D))\in\ZZ\times\ZZ^d.
\]
This is exactly the valuation considered in \cite{LM}.
It may be obtained from the valuation $(\fv_\F)_D$ considered above by applying an invertible linear transformation. 
The advantage of this modified valuation $\fv_{\F,D}'$ is that the Newton-Okounkov body $\Delta(R(\CO(D)),\fv_{\F,D}')$ depends only on the numerical equivalence class of $D$ \cite[Proposition 4.1]{LM}.

\begin{rem}
Given a projective variety $X$, flag $\F$ as above, and ample divisor $D$, Anderson obtains a degeneration of $X$ to 
\[\proj \KK[S\left(R(\CO(D)),\fv_{\F,D}'\right)]\]
provided that $R(\CO(D))$ has a finite Khovanskii basis with respect to the valuation $\fv_{\F,D}'$ \cite{anderson}. A careful reading of \cite[\S5]{anderson} shows that Anderson's degeneration is the Gr\"obner degeneration associated to a weight vector $w$ he constructs using a method which goes back to Caldero.
In particular, if $X$ is a rational complexity-one $T$-variety and the Khovanskii basis used in Anderson's construction is a set of semi-canonical generators for $R(\CO(D))$, then Anderson's degeneration agrees with one of the Gr\"obner degenerations we've discussed in \S\ref{sec:trop}. 
\end{rem}

\subsection{Newton-Okounkov Bodies for Rational Complexity-One $T$-Varieties}\label{sec:petersen}
Let $X$ be a projective rational complexity-one $T$-variety and $D$ a $T$-invariant divisor. Then $\spec R(\CO(D))$ is an affine rational complexity-one $T\times\KK^*$-variety. The polyhedral divisor $\D=\sum_{i=0}^m \Delta_i P_i$ encoding it is described in \cite[Example 2.5]{polarized}. Note that in the rational complexity-one situation, the description there applies to arbitrary invariant divisors, not just very ample ones, after one applies the correction found in \S6 of loc.~cit. The polyhedral divisor $\D$ (and its tailcone $\sigma$) live in $(\ZZ\times N)_\QQ$; we set $\Box_D=\{u\in M_\RR\ | (1,u)\in \sigma^\vee\}$.

Now, if $\fv$ is a valuation on $\KK(X)$ which is homogeneous with respect to the grading of the character group $M$ of $T$, then the valuation $\fv_D$ on $R(\CO(D))$ is $\ZZ\times M$-homogeneous, and the results of \S\ref{sec:homog} apply.

\begin{thm}\label{thm:no}
Let $\fv=\fv_{(\psi,Q,\gamma)}$ be an $M$-homogeneous valuation on $\KK(X)$ with $Q=Q_j$. Then the Newton-Okounkov body of $D$ with respect to $\fv$ is the image of 
\begin{equation}
\left\{ (u,v)\in \Box_D\times \RR\ \big |\  -   \Delta_j(1,u) 
\leq v \leq  \sum_{i\neq j} \Delta_i(1,u)  \right\}
\end{equation}
under the map $\rho:(M\times \ZZ)_\RR \to \Gamma_\RR$ given by $(u,\lambda)\mapsto \psi(u)+\lambda\cdot \gamma$.
\end{thm}
\begin{proof}
  The valuation $\fv_D$ is the $\ZZ\times M$-homogeneous valuation $\fv=\fv_{(\psi',Q,(0,\gamma))}$
with $\psi'(k,u)=(k,\psi(u))$ for $k,u\in \ZZ\times M$. The claim now follows from Corollary \ref{cor:cone}.
  \end{proof}

  In \cite{petersen}, Petersen describes Newton-Okounkov bodies for invariant divisors on projective complexity-one $T$-varieties $X$ with respect to valuations $\fv_\F$ coming from $T$-invariant flags. Let $\pi:X\dashrightarrow \PP^1$ be the rational quotient map, and let $j$ be the largest index for which $\pi(\F_j)$ is a point $Q\in \PP^1$. Then a local computation shows that $\fv_\F=\fv_{(\psi,Q,\gamma)}$ for some $\psi,\gamma$ as in \S\ref{sec:homog}. Having thus determined $Q$, Theorem \ref{thm:no} allows us to describe $\Delta(D,\fv_\F)$, up to linear isomorphism. Comparing with \cite[Theorem 3.9, Propositions 3.13 and 3.15]{petersen}, we see that we recover the same Newton-Okounkov bodies described by Petersen, up to linear isomorphism.

  \subsection{Global Newton-Okounkov Bodies}\label{sec:global}
  Let $X$ be a smooth $d$-dimensional projective variety, and $\F$ a flag as in \S\ref{sec:flag}.
  Fix a basis $D_1,\ldots,D_r$ of divisors of $N^1(X)$, the group of divisors on $X$ modulo numerical equivalence.

  \begin{defn}\cite[\S 4]{LM}
	  The \emph{global Newton-Okounkov body} $\Delta(X,\F)$ of $X$ is the closure  in $\RR^d\times\RR^r$ of the positive hull of all points in $\ZZ^d\times \ZZ^r$ of the form 
	  \[\left(\fv_{\F,D}'(H^0(X,\CO(D))),u\right),\]
	  where $u\in \ZZ^r$ and $D=\sum_i u_iD_i$.
  \end{defn}
  The cone $\Delta(X,\F)$ is a cone projecting onto the pseudo-effective cone of $X$. This cone $\Delta(X,\F)$ is independent of the choice of basis of $N^1(X)$, and for any big class $\xi\in N^1(X)$, the Newton-Okounkov body of $\xi$ (with respect to $\F$) is the fiber over $\xi$ of this projection \cite[Theorem 4.5]{LM}.

  In \cite[Theorem 5.1]{petersen}, Petersen shows that $\Delta(X,\F)$ is rational polyhedral whenever $X$ is a complexity-one $T$-variety. We readily recover that result here under the additional assumption that $X$ is rational: 
  \begin{thm}
For any smooth projective rational complexity-one $T$-variety $X$ and any flag $\F$ as in \S\ref{sec:flag}, the global Newton-Okounkov body $\Delta(X,\F)$ is rational polyhedral.
  \end{thm}

  \begin{proof}
	  For a smooth projective rational complexity-one $T$-variety, $N^1(X)\cong \Pic(X)\cong \Cl(X)$, so we may take $D_1,\ldots, D_r$ as a $\ZZ$-basis of $\Cl(X)$. The \emph{Cox ring} of $X$ is the graded ring
\[
	\cox(X)=\bigoplus_{u\in\ZZ^r} H^0\left(X,\CO(\sum_i u_iD_i)\right)\cdot\chi^u;
\]
it is finitely generated, see e.g. \cite{hausen}. In fact, $\spec \cox(X)$ is an affine, rational complexity-one $T$-variety, see Remark \ref{rem:cox}.

Then by construction, $\Delta(X,\F)$ is just the cone $C(\cox(X),\fv)$ for the homogeneous valuation sending $s\cdot \chi^u$ to $(\fv'_{\F,D}(s),u)$, where $s\in H^0(X,\CO(D))$ for $D=\sum u_iD_i$. But $C(\cox(X),\fv)$ is the closure of the positive hull of a finitely generated semigroup by Theorem \ref{thm:khovanskii}. Hence, it is a rational polyhedral cone. 
\end{proof}

\section{Test Configurations and Degenerations}\label{sec:test}
Let $X$ be a projective variety, and $\mcL$ an ample line bundle on $X$. A \emph{test configuration} for the pair $(X,\mcL)$ is a $\KK^*$-equivariant flat family $\widetilde{X}$ over $\Aff^1$ equipped with a relatively ample equivariant $\QQ$-line bundle $\widetilde{\mcL}$ such that
\begin{enumerate}
	\item The $\KK^*$-action on $(\widetilde{X},\widetilde{\mcL})$ lifts the standard action on $\Aff^1$;
	\item The general fiber is isomorphic to $X$, with $\widetilde{\mcL}$ restricting to $\mcL$.
\end{enumerate}
Such test configurations are used in the definition of K-stability and the study of the existence of K\"ahler-Einstein metrics on Fano manifolds, see e.g. \cite{donaldson}.

Suppose that the variety in question has an action by some algebraic group $G$. We say that a test configuration $(\widetilde{X},\widetilde{\mcL})$ is $G$-equivariant if $(\widetilde{X},\widetilde{\mcL})$ is equipped with a $G$-action extending the action on $X$, and commuting with the $\KK^*$-action. In \cite[Proposition 4.2 and Theorem 4.3]{kstab}, S\"u\ss{} and the first author classified all normal irreducible special fibers appearing in $T$-equivariant test configurations for projective rational complexity-one $T$-varieties. This led to an effective criterion for determining the existence of a K\"ahler-Einstein metric on a Fano complexity-one $T$-variety \cite[Theorem 4.10]{kstab}.

The approach of \cite{kstab} to the classification of normal special fibers of test configurations utilized the fact that the total space of such a test configuration is itself a rational complexity-one $T$-variety. If one drops the normality assumption for the special fiber, then the total space must also no longer be normal and these methods do not apply. However, we can use our results on homogeneous valuations here to deal with non-normal special fibers:

Let $(X,\mcL)$ be a polarized projective rational complexity-one $T$-variety, and let $L\subset \PP^m$ and $\Delta_0,\ldots,\Delta_m \subset (N\times \ZZ)_\RR$ be such that 
\[
X(L)\cong \spec R(\mcL).
\]

\begin{thm}\label{thm:degen}
	For any $T$-equivariant test configuration $\widetilde{X}$ with irreducible and reduced special fiber $X_0$, either $X_0\cong X$ or $X_0$ is a toric variety of the form $\proj \KK[S]$, where 
\begin{equation*}
S=\left\{ (u,v)\in (\sigma^\vee\cap M\times \ZZ)\times \ZZ\ \big |\  -  \lfloor \Delta_j(u) \rfloor
\leq v \leq  \sum_{i\neq j} \lfloor \Delta_i(u) \rfloor \right\}
\end{equation*}
for some $0\leq j \leq r$, or
\begin{equation*}
S=\left\{ (u,v)\in (\sigma^\vee\cap M\times \ZZ)\times \ZZ\ \big |\  0
\leq v \leq  \sum_{i} \lfloor \Delta_i(u) \rfloor \right\}.
\end{equation*}
Here, the $\ZZ$-grading for $\proj$ is given by the first $\ZZ$-factor in $M\times \ZZ\times \ZZ$.
\end{thm}
\begin{proof}
If $X_0$ is not isomorphic to $X$, then $T\times \KK^*$ must act effectively on $X_0$; since $X_0$ is irreducible and reduced, it must be a toric variety.

	By considering a sufficiently high multiple of $\widetilde{\mcL}$, the test configuration $\widetilde X$ can be embedded equivariantly in some $\PP^n\times\Aff^1$. If $I$ is the ideal of $X$ in $\PP^n$, then the special fiber $X_0$ is given by the initial ideal $\In_w(I)$ with regards to some weight $w$ determined by the $\KK^*$-action.
	
	Since $X_0$ is reduced by assumption, $\In_w(I)$ must be prime. In particular, it can only contain a monomial if it contains a variable, in which case $I$ must have contained a linear form. In this case, $X$ is contained in a smaller projective space, so after eventually passing to a linear subspace of $\PP^n$, we can assume that $\In_w(X)$ contains no monomials, that is, $w\in \trop(X^\circ)$. In fact, $w$ is in a prime cone $C$ of $\trop(X^\circ)$. 

	Let $R$ be the homogeneous coordinate ring of $X$.	Now, following \cite[Theorem 4]{Kaveh-Manon-NOK}, one can use the prime cone $C$ to construct a homogeneous weight valuation $\fv$ on $R$ such that $X_0$ is $\proj$ of the semigroup algebra for $S(R,\fv)$. But since $\fv$ is a homogeneous valuation, we may apply Theorem \ref{thm:khovanskii} for a concrete description of $S(R,\fv)$. 
\end{proof}

\begin{ex}We continue Examples \ref{ex:proj} and \ref{ex:proj2}. Since all coefficients $\Delta_0,\Delta_1,\Delta_2$ are lattice polytopes, it follows from Theorem \ref{thm:degen}  that for any $T$-equivariant non-trivial test configuration of $X=\PP(\Omega_{\PP^2})$ with integral special fiber, that special fiber is a \emph{normal} toric variety. Up to isomorphism, only two possibilities occur. The first is the variety cut out by the $2\times 2$ minors of 
\[
\left(\begin{array}{c c c}
x_0&x_4&x_7\\
x_1&-x_0&x_8\\
x_2&x_5&x_6
\end{array}\right)
\]
as described before.
The second (corresponding to the case where $\fv$ involves a point $Q\neq Q_i$) is cut out by the $2\times 2$ minors of 
\[
\left(\begin{array}{c c c}
x_0&x_4&x_7\\
x_1&x_0&x_8\\
x_2&x_5&x_0
\end{array}\right).
\]
We cannot see this second kind of degeneration as a Gr\"obner degeneration with regards to the embedding of Example \ref{ex:proj}, but must consider a semi-canonical embedding for $X$ involving a fourth point on $\PP^1$ apart from $0,1,\infty$.
  \end{ex}

\begin{rem}
	Our Theorem \ref{thm:degen} considers a larger class of degenerations than those considered by S\"u\ss{} and the first author in \cite{kstab}, namely, it includes test configurations with non-normal special fibers. However, it does not shed any new light on K-stability or existence of K\"ahler-Einstein metrics for rational complexity-one Fano $T$-varieties. These questions were already completely dealt with in \cite{kstab}, since one only needs to consider test configurations with normal special fibers.
\end{rem}

\subsection*{Acknowledgements} The idea of embedding a rational complexity-one $T$-variety in this fashion dates back to Altmann, Hausen, and S\"u\ss{}. We thank J\"urgen Hausen, Lars Kastner, Diane Maclagan, and  Hendrik S\"u\ss{} for helpful discussions. We acknowledge the Fields Institute and the program \emph{Combinatorial Algebraic Geometry} for support. We also thank the anonymous referees for useful feedback.
\bibliographystyle{alpha}
\bibliography{poised} 

\newcommand{\etalchar}[1]{$^{#1}$}
\begin{thebibliography}{HMSV09}

\bibitem[ABHW]{iteration}
Ivan Arzhantsev, Lukas Braun, J\"urgen Hausen, and Milena Wrobel.
\newblock Log terminal singularities, platonic tuples and iteration of {C}ox
  rings.
\newblock arXiv:1703.03627v1 [math.AG].

\bibitem[ADHL15]{coxrings}
Ivan Arzhantsev, Ulrich Derenthal, J{\"u}rgen Hausen, and Antonio Laface.
\newblock {\em Cox rings}, volume 144 of {\em Cambridge Studies in Advanced
  Mathematics}.
\newblock Cambridge University Press, Cambridge, 2015.

\bibitem[AH06]{pdiv}
Klaus Altmann and J{\"u}rgen Hausen.
\newblock Polyhedral divisors and algebraic torus actions.
\newblock {\em Math. Ann.}, 334(3):557--607, 2006.

\bibitem[AHS08]{dfan}
Klaus Altmann, J{\"u}rgen Hausen, and Hendrik S{\"u}ss.
\newblock Gluing affine torus actions via divisorial fans.
\newblock {\em Transform. Groups}, 13(2):215--242, 2008.

\bibitem[AIP{\etalchar{+}}12]{tsurvey}
Klaus Altmann, Nathan~Owen Ilten, Lars Petersen, Hendrik S{\"u}{\ss}, and
  Robert Vollmert.
\newblock The geometry of {$T$}-varieties.
\newblock In {\em Contributions to algebraic geometry}, EMS Ser. Congr. Rep.,
  pages 17--69. Eur. Math. Soc., Z\"urich, 2012.

\bibitem[Alt95]{altmann:95a}
Klaus Altmann.
\newblock Minkowski sums and homogeneous deformations of toric varieties.
\newblock {\em Tohoku Math. J. (2)}, 47(2):151--184, 1995.

\bibitem[And13]{anderson}
Dave Anderson.
\newblock Okounkov bodies and toric degenerations.
\newblock {\em Math. Ann.}, 356(3):1183--1202, 2013.

\bibitem[AP12]{pcox}
Klaus Altmann and Lars Petersen.
\newblock Cox rings of rational complexity-one {$T$}-varieties.
\newblock {\em J. Pure Appl. Algebra}, 216(5):1146--1159, 2012.

\bibitem[BGT97]{koszul}
Winfried Bruns, Joseph Gubeladze, and Ng\^o~Vi\^et Trung.
\newblock Normal polytopes, triangulations, and {K}oszul algebras.
\newblock {\em J. Reine Angew. Math.}, 485:123--160, 1997.

\bibitem[CDS14]{donaldson}
Xiuxiong Chen, Simon Donaldson, and Song Sun.
\newblock K\"ahler-{E}instein metrics and stability.
\newblock {\em Int. Math. Res. Not. IMRN}, (8):2119--2125, 2014.

\bibitem[CLO15]{CLO}
David~A. Cox, John Little, and Donal O'Shea.
\newblock {\em Ideals, varieties, and algorithms}.
\newblock Undergraduate Texts in Mathematics. Springer, Cham, fourth edition,
  2015.
\newblock An introduction to computational algebraic geometry and commutative
  algebra.

\bibitem[CLS11]{CLS}
David~A. Cox, John~B. Little, and Henry~K. Schenck.
\newblock {\em Toric varieties}, volume 124 of {\em Graduate Studies in
  Mathematics}.
\newblock American Mathematical Society, Providence, RI, 2011.

\bibitem[DP16]{Draisma-Postinghel}
Jan Draisma and Elisa Postinghel.
\newblock Faithful tropicalisation and torus actions.
\newblock {\em Manuscripta Math.}, 149(3-4):315--338, 2016.

\bibitem[GRW16]{Gubler-Rabinoff-Werner}
Walter Gubler, Joseph Rabinoff, and Annette Werner.
\newblock Skeletons and tropicalizations.
\newblock {\em Adv. Math.}, 294:150--215, 2016.

\bibitem[HMM11]{Howard-Manon-Millson}
Benjamin Howard, Christopher Manon, and John Millson.
\newblock The toric geometry of triangulated polygons in {E}uclidean space.
\newblock {\em Canad. J. Math.}, 63(4):878--937, 2011.

\bibitem[HMSV09]{Howard-Millson-Snowden-Vakil}
Benjamin Howard, John Millson, Andrew Snowden, and Ravi Vakil.
\newblock The equations for the moduli space of {$n$} points on the line.
\newblock {\em Duke Math. J.}, 146(2):175--226, 2009.

\bibitem[HS10]{hausen}
J{\"u}rgen Hausen and Hendrik S{\"u}{\ss}.
\newblock The {C}ox ring of an algebraic variety with torus action.
\newblock {\em Adv. Math.}, 225(2):977--1012, 2010.

\bibitem[IS11]{polarized}
Nathan~Owen Ilten and Hendrik S\"uss.
\newblock Polarized complexity-1 {$T$}-varieties.
\newblock {\em Michigan Math. J.}, 60(3):561--578, 2011.

\bibitem[IS17]{kstab}
Nathan Ilten and Hendrik S\"uss.
\newblock K-stability for {F}ano manifolds with torus action of complexity 1.
\newblock {\em Duke Math. J.}, 166(1):177--204, 2017.

\bibitem[IV12]{tdef}
Nathan~Owen Ilten and Robert Vollmert.
\newblock Deformations of rational {$T$}-varieties.
\newblock {\em J. Algebraic Geom.}, 21(3):531--562, 2012.

\bibitem[KK12]{KK}
Kiumars Kaveh and A.~G. Khovanskii.
\newblock Newton-{O}kounkov bodies, semigroups of integral points, graded
  algebras and intersection theory.
\newblock {\em Ann. of Math. (2)}, 176(2):925--978, 2012.

\bibitem[KM]{Kaveh-Manon-NOK}
Kiumars Kaveh and Christopher Manon.
\newblock Khovanskii bases, {N}ewton-{O}kounkov polytopes and tropical geometry
  of projective varieties.
\newblock arXiv:1610.00298v3 [math.AG].

\bibitem[Lan15]{langlois}
Kevin Langlois.
\newblock Polyhedral divisors and torus actions of complexity one over
  arbitrary fields.
\newblock {\em J. Pure Appl. Algebra}, 219(6):2015--2045, 2015.

\bibitem[LM09]{LM}
Robert Lazarsfeld and Mircea Musta\c{t}\u{a}.
\newblock Convex bodies associated to linear series.
\newblock {\em Ann. Sci. \'Ec. Norm. Sup\'er. (4)}, 42(5):783--835, 2009.

\bibitem[MM86]{mori}
Shigefumi Mori and Shigeru Mukai.
\newblock Classification of {F}ano {$3$}-folds with {$B_2\geq 2$}. {I}.
\newblock In {\em Algebraic and topological theories ({K}inosaki, 1984)}, pages
  496--545. Kinokuniya, Tokyo, 1986.

\bibitem[MS15]{tropical}
Diane Maclagan and Bernd Sturmfels.
\newblock {\em Introduction to tropical geometry}, volume 161 of {\em Graduate
  Studies in Mathematics}.
\newblock American Mathematical Society, Providence, RI, 2015.

\bibitem[NNU12]{Nishinou-Nohara-Ueda}
Takeo Nishinou, Yuichi Nohara, and Kazushi Ueda.
\newblock Potential functions via toric degenerations.
\newblock {\em Proc. Japan Acad. Ser. A Math. Sci.}, 88(2):31--33, 2012.

\bibitem[Pee11]{syzygies}
Irena Peeva.
\newblock {\em Graded syzygies}, volume~14 of {\em Algebra and Applications}.
\newblock Springer-Verlag London, Ltd., London, 2011.

\bibitem[Pet]{petersen}
Lars Petersen.
\newblock Ok\-oun\-kov bodies of complexity-one {T}-varieties.
\newblock arXiv:1108.0632v1 [math.AG].

\bibitem[PRS98]{quadratic}
Irena Peeva, Victor Reiner, and Bernd Sturmfels.
\newblock How to shell a monoid.
\newblock {\em Math. Ann.}, 310(2):379--393, 1998.

\bibitem[SS04]{Speyer-Sturmfels}
David Speyer and Bernd Sturmfels.
\newblock The tropical {G}rassmannian.
\newblock {\em Adv. Geom.}, 4(3):389--411, 2004.

\bibitem[Stu96]{Sturmfels-GBCP}
Bernd Sturmfels.
\newblock {\em Gr\"obner bases and convex polytopes}, volume~8 of {\em
  University Lecture Series}.
\newblock American Mathematical Society, Providence, RI, 1996.

\bibitem[S{\"u}s14]{picturebook}
Hendrik S{\"u}ss.
\newblock Fano threefolds with 2-torus action: a picture book.
\newblock {\em Doc. Math.}, 19:905--940, 2014.

\bibitem[ZS75]{ZS}
Oscar Zariski and Pierre Samuel.
\newblock {\em Commutative algebra. {V}ol. {II}}.
\newblock Springer-Verlag, New York-Heidelberg, 1975.
\newblock Reprint of the 1960 edition, Graduate Texts in Mathematics, Vol. 29.

\end{thebibliography}
\end{document}